\documentclass[a4paper,10pt]{article}

\usepackage[english,british]{babel}
\usepackage{amsmath,amsthm,amssymb,dsfont,CJK,esint}
\usepackage{adjustbox,lipsum}
\usepackage{tikz}
\usepackage{standalone}
\usepackage{mathtools}
\usepackage{fancyhdr}
\usepackage{graphicx}
\usepackage{bibentry}
\usepackage{enumitem}
\usepackage{xcolor}
\usepackage{stackengine}
\usepackage[
            pdfstartview=FitH,
            CJKbookmarks=true,
            bookmarksnumbered=true,
            bookmarksopen=true,
            pdfborder=001,
            ]{hyperref}
\usepackage[a4paper,width=160mm,top=27mm,bottom=27mm]{geometry}
\hypersetup{
}
\mathtoolsset{showonlyrefs=true}

\newtheorem{theorem}{Theorem}[section]
\newtheorem{lemma}[theorem]{Lemma}

\newtheorem{proposition}[theorem]{Proposition}
\newtheorem{corollary}[theorem]{Corollary}

\theoremstyle{definition}
\newtheorem{remx}[theorem]{Remark}
\newenvironment{remark}
  {\pushQED{\qed}\remx}
  {\popQED\endremx}

\makeatletter
\newsavebox\myboxA
\newsavebox\myboxB
\newlength\mylenA

\newcommand*\yoverline[2][0.75]{%
    \sbox{\myboxA}{$\m@th#2$}%
    \setbox\myboxB\null
    \ht\myboxB=\ht\myboxA%
    \dp\myboxB=\dp\myboxA%
    \wd\myboxB=#1\wd\myboxA
    \sbox\myboxB{$\m@th\overline{\copy\myboxB}$}
    \setlength\mylenA{\the\wd\myboxA}
    \addtolength\mylenA{-\the\wd\myboxB}%
    \ifdim\wd\myboxB<\wd\myboxA%
       \rlap{\hskip 0.5\mylenA\usebox\myboxB}{\usebox\myboxA}%
    \else
        \hskip -0.5\mylenA\rlap{\usebox\myboxA}{\hskip 0.5\mylenA\usebox\myboxB}%
    \fi}
\makeatother

\numberwithin{equation}{section}

\begin{document}




\allowdisplaybreaks

\newcommand{\diver}{\operatorname{div}}
\newcommand{\lin}{\operatorname{Lin}}
\newcommand{\curl}{\operatorname{curl}}
\newcommand{\ran}{\operatorname{Ran}}
\newcommand{\kernel}{\operatorname{Ker}}
\newcommand{\la}{\langle}
\newcommand{\ra}{\rangle}
\newcommand{\N}{\mathbb{N}}
\newcommand{\R}{\mathbb{R}}
\newcommand{\C}{\mathbb{C}}
\newcommand{\T}{\mathbb{T}}

\newcommand{\ld}{\lambda}
\newcommand{\fai}{\varphi}
\newcommand{\0}{0}
\newcommand{\n}{\mathbf{n}}
\newcommand{\uu}{{\boldsymbol{\mathrm{u}}}}
\newcommand{\UU}{{\boldsymbol{\mathrm{U}}}}
\newcommand{\buu}{\bar{{\boldsymbol{\mathrm{u}}}}}
\newcommand{\ten}{\\[4pt]}
\newcommand{\six}{\\[-3pt]}
\newcommand{\nb}{\nonumber}
\newcommand{\hgamma}{H_{\Gamma}^1(\OO)}
\newcommand{\opert}{O_{\varepsilon,h}}
\newcommand{\barx}{\bar{x}}
\newcommand{\barf}{\bar{f}}
\newcommand{\hatf}{\hat{f}}
\newcommand{\xoneeps}{x_1^{\varepsilon}}
\newcommand{\xh}{x_h}
\newcommand{\scaled}{\nabla_{1,h}}
\newcommand{\scaledb}{\widehat{\nabla}_{1,\gamma}}
\newcommand{\vare}{\varepsilon}
\newcommand{\A}{{\bf{A}}}
\newcommand{\RR}{{\bf{R}}}
\newcommand{\B}{{\bf{B}}}
\newcommand{\CC}{{\bf{C}}}
\newcommand{\D}{{\bf{D}}}
\newcommand{\K}{{\bf{K}}}
\newcommand{\oo}{{\bf{o}}}
\newcommand{\id}{{\bf{Id}}}
\newcommand{\E}{\mathcal{E}}
\newcommand{\ii}{\mathcal{I}}
\newcommand{\sym}{\mathrm{sym}}
\newcommand{\lt}{\left}
\newcommand{\rt}{\right}
\newcommand{\ro}{{\bf{r}}}
\newcommand{\so}{{\bf{s}}}
\newcommand{\e}{{\bf{e}}}
\newcommand{\ww}{{\boldsymbol{\mathrm{w}}}}
\newcommand{\zz}{{\boldsymbol{\mathrm{z}}}}
\newcommand{\U}{{\boldsymbol{\mathrm{U}}}}
\newcommand{\G}{{\boldsymbol{\mathrm{G}}}}
\newcommand{\VV}{{\boldsymbol{\mathrm{V}}}}
\newcommand{\II}{{\boldsymbol{\mathrm{I}}}}
\newcommand{\ZZ}{{\boldsymbol{\mathrm{Z}}}}
\newcommand{\hKK}{{{\bf{K}}}}
\newcommand{\f}{{\bf{f}}}
\newcommand{\g}{{\bf{g}}}
\newcommand{\lkk}{{\bf{k}}}
\newcommand{\tkk}{{\tilde{\bf{k}}}}
\newcommand{\W}{{\boldsymbol{\mathrm{W}}}}
\newcommand{\Y}{{\boldsymbol{\mathrm{Y}}}}
\newcommand{\EE}{{\boldsymbol{\mathrm{E}}}}
\newcommand{\F}{{\bf{F}}}
\newcommand{\spacev}{\mathcal{V}}
\newcommand{\spacevg}{\mathcal{V}^{\gamma}(\Omega\times S)}
\newcommand{\spacevb}{\bar{\mathcal{V}}^{\gamma}(\Omega\times S)}
\newcommand{\spaces}{\mathcal{S}}
\newcommand{\spacesg}{\mathcal{S}^{\gamma}(\Omega\times S)}
\newcommand{\spacesb}{\bar{\mathcal{S}}^{\gamma}(\Omega\times S)}
\newcommand{\skews}{H^1_{\barx,\mathrm{skew}}}
\newcommand{\kk}{\mathcal{K}}
\newcommand{\OO}{O}
\newcommand{\bhe}{{\bf{B}}_{\vare,h}}
\newcommand{\pp}{{\mathbb{P}}}
\newcommand{\ff}{{\mathcal{F}}}
\newcommand{\mWk}{{\mathcal{W}}^{k,2}(\Omega)}
\newcommand{\mWa}{{\mathcal{W}}^{1,2}(\Omega)}
\newcommand{\mWb}{{\mathcal{W}}^{2,2}(\Omega)}
\newcommand{\twos}{\xrightharpoonup{2}}
\newcommand{\twoss}{\xrightarrow{2}}
\newcommand{\bw}{\bar{w}}
\newcommand{\br}{\bar{{\bf{r}}}}
\newcommand{\bz}{\bar{{\bf{z}}}}
\newcommand{\tw}{{W}}
\newcommand{\tr}{{{\bf{R}}}}
\newcommand{\tz}{{{\bf{Z}}}}
\newcommand{\lo}{{{\bf{o}}}}
\newcommand{\hoo}{H^1_{00}(0,L)}
\newcommand{\ho}{H^1_{0}(0,L)}
\newcommand{\hotwo}{H^1_{0}(0,L;\R^2)}
\newcommand{\hooo}{H^1_{00}(0,L;\R^2)}
\newcommand{\hhooo}{H^1_{00}(0,1;\R^2)}
\newcommand{\dsp}{d_{S}^{\bot}(\barx)}
\newcommand{\LB}{{\bf{\Lambda}}}
\newcommand{\LL}{\mathbb{L}}
\newcommand{\mL}{\mathcal{L}}
\newcommand{\mhL}{\widehat{\mathcal{L}}}
\newcommand{\loc}{\mathrm{loc}}
\newcommand{\tqq}{\mathcal{Q}^{*}}
\newcommand{\tii}{\mathcal{I}^{*}}
\newcommand{\Mts}{\mathbb{M}}
\newcommand{\pot}{\mathrm{pot}}
\newcommand{\tU}{{\widehat{\bf{U}}}}
\newcommand{\tVV}{{\widehat{\bf{V}}}}
\newcommand{\pt}{\partial}
\newcommand{\bg}{\Big}
\newcommand{\hA}{\widehat{{\bf{A}}}}
\newcommand{\hB}{\widehat{{\bf{B}}}}
\newcommand{\hCC}{\widehat{{\bf{C}}}}
\newcommand{\hD}{\widehat{{\bf{D}}}}
\newcommand{\fder}{\partial^{\mathrm{MD}}}
\newcommand{\Var}{\mathrm{Var}}
\newcommand{\pta}{\partial^{0\bot}}
\newcommand{\ptaj}{(\partial^{0\bot})^*}
\newcommand{\ptb}{\partial^{1\bot}}
\newcommand{\ptbj}{(\partial^{1\bot})^*}
\newcommand{\geg}{\Lambda_\vare}
\newcommand{\tpta}{\tilde{\partial}^{0\bot}}
\newcommand{\tptb}{\tilde{\partial}^{1\bot}}
\newcommand{\ua}{u_\alpha}
\newcommand{\pa}{p\alpha}
\newcommand{\qa}{q(1-\alpha)}
\newcommand{\Qa}{Q_\alpha}
\newcommand{\Qb}{Q_\eta}
\newcommand{\ga}{\gamma_\alpha}
\newcommand{\gb}{\gamma_\eta}
\newcommand{\ta}{\theta_\alpha}
\newcommand{\tb}{\theta_\eta}


\newcommand{\mH}{\mathcal{E}}
\newcommand{\mD}{\mathcal{D}}
\newcommand{\csob}{\mathcal{S}}
\newcommand{\mA}{\mathcal{A}}
\newcommand{\mK}{\mathcal{K}}
\newcommand{\mS}{\mathcal{S}}
\newcommand{\mI}{\mathcal{I}}
\newcommand{\tas}{{2_*}}
\newcommand{\tbs}{{2^*}}
\newcommand{\tm}{{\tilde{m}}}
\newcommand{\tdu}{{\phi}}
\newcommand{\tpsi}{{\tilde{\psi}}}
\newcommand{\Z}{{\mathbb{Z}}}
\newcommand{\tsigma}{{\tilde{\sigma}}}
\newcommand{\tg}{{\tilde{g}}}
\newcommand{\tG}{{\tilde{G}}}
\newcommand{\mM}{\mathcal{M}}
\newcommand{\mC}{\mathcal{C}}
\newcommand{\wlim}{{\text{w-lim}}\,}
\newcommand{\diag}{L_{t,x}^4}
\newcommand{\vu}{ u}
\newcommand{\vz}{ z}
\newcommand{\vv}{ v}
\newcommand{\ve}{ e}
\newcommand{\vw}{ w}
\newcommand{\vf}{ f}
\newcommand{\vh}{ h}
\newcommand{\vp}{ \vec P}
\newcommand{\ang}{{\not\negmedspace\nabla}}
\newcommand{\dxy}{\Delta_{x,y}}
\newcommand{\lxy}{L_{x,y}}
\newcommand{\gnsand}{\mathrm{C}_{\mathrm{GN},3d}}

\title{Large data global well-posedness and scattering for the focusing cubic nonlinear Schr\"odinger equation on $\mathbb{R}^2\times\mathbb{T}$}
\author{Yongming Luo \thanks{Institut f\"{u}r Wissenschaftliches Rechnen, Technische Universit\"at Dresden, Germany} \thanks{\href{mailto:yongming.luo@tu-dresden.de}{Email: yongming.luo@tu-dresden.de}}
}

\date{}
\maketitle

\begin{abstract}
We consider the focusing cubic nonlinear Schr\"odinger equation (NLS)
\begin{align}\label{CNLSS}
i\partial_t U+\Delta U=-|U|^2U\quad\text{on $\mathbb{R}^2\times\mathbb{T}$}.\tag{3NLS}
\end{align}
Different from the 3D Euclidean case, the \eqref{CNLSS} is mass-critical and non-scale-invariant on the waveguide manifold $\mathbb{R}^2\times\mathbb{T}$, hence the underlying analysis becomes more subtle and challenging. We formulate thresholds using the 2D Euclidean ground state of the focusing cubic NLS and show that solutions of \eqref{CNLSS} lying below the thresholds are global and scattering in time. The proof relies on several new established Gagliardo-Nirenberg inequalities, whose best constants are formulated in term of the 2D Euclidean ground state. It is also worth noting the interesting fact that the thresholds for global well-posedness and scattering do not coincide. To the author's knowledge, this paper also gives the first large data scattering result for focusing NLS on product spaces.
\end{abstract}

\tableofcontents


\section{Introduction}\label{sec:Introduction}
In this paper, we study the focusing cubic nonlinear Schr\"odinger equation (NLS)
\begin{align}
i\pt_t U+\Delta_{x,y} U=-|U|^2U\label{cnls}
\end{align}
on the waveguide manifold $\R_x^2\times \T_y$. The equation \eqref{cnls} arises from various domains in applied sciences such as nonlinear optics and Bose-Einstein condensation. We refer to \cite{waveguide_ref_1,waveguide_ref_2,waveguide_ref_3} for a detailed introduction on the physical background of \eqref{cnls}. It is worth noting at this point that although large data problems for the defocusing analogues of \eqref{cnls} have been extensively studied \cite{HaniPausader,Yang_Zhao_2018,R1T1Scattering,RmT1,R2T2}, results concerning the focusing models are relatively less well-known. The purpose of this paper is to give a first step towards the large data scattering\footnote{We mainly focus on the (much harder) scattering problem. In fact, due to the energy-subcritical nature of \eqref{cnls} a large data global well-posedness result follows immediately from Lemma \ref{holmer variational} given below. See Theorem \ref{thm gwp} for details.} for NLS of focusing type on product spaces.

A first systematic study on NLS on compact manifolds might at least date back to \cite{BrezisGallouet1980}, where the authors studied the cubic NLS on a bounded domain or an exterior domain in $\R^2$. Concerning NLS on tori, Bourgain studied in his seminal papers \cite{Bourgain1,Bourgain2} the NLS and KdV equations on rational tori. Particularly, using number theoretical methods Bourgain proved (endpoint and non-endpoint) Strichartz estimates for NLS and KdV on rational tori, which were also utilized to establish different local and global well-posedness results. As a byproduct of the proof of the $\ell^2$-decoupling conjecture \cite{l2decoupling}, the Strichartz estimates on rational tori were later extended to irrational ones, including both endpoint and non-endpoint cases.

For NLS on more general compact manifolds, a systematic study was initiated in a series of works \cite{Burq1,Burq2,Burq3} by Burq, G\'{e}rard and Tzvetkov, where the authors proved Strichartz and multilinear estimates, local and global well-posedness results for NLS on compact manifolds. Using the theory of atomic spaces initiated by Koch and Tataru \cite{KochTataru}, Herr, Tataru, and Tzvetkov were able to prove local and global well-posedness results for NLS on tori, product spaces and Zoll-manifolds in the energy-critical case \cite{HerrZoll,HerrTataruTz1,HerrTataruTz2}. However, at the energy-critical level the well-posedness results also depend on the profile of the initial data, and a large data result can in general not be obtained using only the \textit{a priori} bounds deduced from the conservation laws. Following the nowadays well-known concentration compactness arguments initiated by Kenig and Merle \cite{KenigMerle2006} and the so-called Black-Box-Theory, Ionescu, Pausader and Staffilani \cite{Ionescu1,Ionescu2,hyperbolic} showed that defocusing energy-critical NLS on $\T^3$, $\R\times\T^3$ and on the hyperbolic space $\mathbb{H}^3$ are always globally well-posed. By appealing to suitable variational arguments, Yu, Yue and Zhao \cite{Haitian,YuYueZhao2021} utilized the Black-Box-Theory to prove that solutions of the focusing energy-critical NLS on $\T^4$ and $\R\times\T^3$ lying below ground states are always globally well-posed.

The above mentioned models can be generalized to the NLS
\begin{align}
i\pt_t U+\Delta_{x,y} U=\pm |U|^\alpha U\label{general cnls}
\end{align}
on the manifold $\R^d\times \mathcal{M}$, where $\mathcal{M}$ is an $n$-dimensional compact manifold. Loosely speaking, the dispersion of an NLS-wave on a compact manifold is much weaker than on $\R^d$, hence it is not expected that scattering takes place for large $n$. Indeed, even in the defocusing case, a global but not scattering solution of NLS on $\T^d$ does exist, see for instance \cite{global_but_not_scatter}. Nonetheless, in view of the classical long time dynamics results for NLS on Euclidean spaces\footnote{We refer an NLS on Euclidean space to an NLS on $\R^d$.}, the rather weak dispersion effect corresponding to the manifold $\mathcal{M}$ can be compensated by the stronger dispersion on $\R^d$, and scattering solutions\footnote{We are referring here to scattering in the $H^1$-energy space, which is the minimal space that admits all useful conservation laws. Such heuristics do not hold when the scattering is considered in spaces of higher differentiation order, see \cite{TNCommPDE}.} are expected when
\begin{itemize}
\item[(i)] The nonlinearity is at most energy-critical w.r.t. the space dimension $d+n$, and
\item[(ii)]The nonlinearity is at least mass-critical w.r.t. the space dimension $d$.
\end{itemize}
In other words, we expect that a general scattering theory as the one formulated in the Euclidean case should be available for $\alpha$ lying in the range $[\frac{4}{d},\frac{4}{d+n-2}]$. Particularly, it is necessary that $n\leq 2$. In this direction, the first contribution was made by Tzvetkov and Visciglia \cite{TNCommPDE}, where the authors studied well-posedness and scattering of solutions of \eqref{general cnls} with small initial data in non-isotropic Sobolev spaces. The same authors studied later in \cite{TzvetkovVisciglia2016} the special case where $\mathcal{M}=\T$ and \eqref{general cnls} is defocusing. Particularly, they proved that \eqref{general cnls} is always globally well-posed for $\alpha\in(0,\frac{4}{d-1})$ and additionally scattering for $\alpha\in(\frac4d,\frac{4}{d-1})$ in $H^1(\R^d\times\T)$. In the case where $\mathcal{M}=\T^n$ and the nonlinearity is mass-critical ($\alpha=\frac{4}{d}$) or energy-critical ($\alpha=\frac{4}{d+n-2}$), the first breakthrough was made by Hani and Pausader \cite{HaniPausader}, where they studied defocusing \eqref{general cnls} with $\alpha=4$, $d=1$ and $n=2$, which is the well-known defocusing quintic NLS on $\R\times\T^2$. Particularly, \eqref{general cnls} in this case is both mass- and energy-critical. Based on a conjecture on large data scattering of the large scale resonant system of \eqref{general cnls}, which was later solved by Cheng, Guo and Zhao \cite{R1T1Scattering}, Hani and Pausader proved that a solution of \eqref{general cnls} is always global and scattering. Utilizing the methodologies of \cite{HaniPausader} for \eqref{general cnls}, the ones of \cite{Yang_Zhao_2018} for the corresponding large scale resonant system of \eqref{general cnls} and the Black-Box-Theory, the large data scattering problem for defocusing \eqref{general cnls} with critically algebraic nonlinearities on $\R^d\times\T^n$ has been completely resolved \cite{HaniPausader,Yang_Zhao_2018,R1T1Scattering,RmT1,R2T2}. We also refer to \cite{ModifiedScattering,TerraciniTzvetkovVisciglia,SystemProdSpace,hari2016small,HariVisciglia2018,FoprcellaHari2020,Barron} for  further interesting topics in this direction.

Let us now focus on the focusing cubic NLS \eqref{cnls} and explain briefly the mass-criticality of \eqref{cnls}. Indeed, we may simply assume that \eqref{cnls} is independent of the $y$-variable, and in this case \eqref{cnls} reduces to the 2D Euclidean focusing cubic NLS, which is known to be mass-critical. In order to incorporate the full impact of $\T$ into the problem, we should instead consider the following scaling transformation heuristics inspired by Hani and Pausader \cite{HaniPausader}: it is easy to verify that \eqref{cnls} remains invariant under the scaling transformation
\begin{align}
U(t,x,y)\mapsto U_\ld(t,x,y):=\ld U(\ld^{2}t,\ld (x,y)).
\end{align}
We should however keep in mind that the occupying domain $\R^2\times \T$ is also deformed to $\R^2\times \ld^{-1} \T$. By sending $\ld\to\infty$\footnote{Since \eqref{cnls} is energy-subcritical, the small scale limit $\ld\to 0$ is irrelevant.} (namely the so-called large scale limit) we see that the deformed torus $\ld^{-1}\T$ becomes thinner and thinner, thus also ignorable in comparison with $\R^2$. To make these heuristics rigorous, we recall that the solution $U$ can be written into the Fourier series (up to some Fourier constant) w.r.t. the $y$-variable
\begin{align}
U(t,x,y)=\sum_{k\in\Z}(\mathcal{F}_y U)(t,x,k)e^{iky}.
\end{align}
Hence we may represent the nonlinear potential $|U|^2U$ by
\begin{align}
(|U|^2U)(y)=\sum_{k}\sum_{(k_1,k_2,k_3)\in \mathcal{I}_k}\mathcal{F}_y U(k_1)\overline{\mathcal{F}_y U}(k_2)
\mathcal{F}_y U(k_3)e^{iky},
\end{align}
where
$$\mathcal{I}_k:=\{(k_1,k_2,k_3)\in\Z^3:k_1-k_2+k_3=k\}.$$
We may further decompose $\mathcal{I}_k$ into the resonant part (RS) and non-resonant part (NRS):
\begin{align}
\mathcal{I}_k=\{(k_1,k_2,k_3)\in\mathcal{I}_k:k_1^2-k_2^2+k_3^2=k^2\}
\cup\{(k_1,k_2,k_3)\in\mathcal{I}_k:k_1^2-k_2^2+k_3^2\neq k^2\}
=:\mathcal{RS}_k\cup\mathcal{NRS}_k.
\end{align}
The idea is as follows: the RS part can be seen as a non-perturbative component that should be dealt in a more complex and serious way. Nevertheless, by applying a normal form transformation, the NRS part will be relaxed in the large scale limitation. We refer to \cite[Lem. 5.7]{HaniPausader} or \cite[Lem. 3.11]{CubicR2T1Scattering} for details of a rigorous verification of such intuitive heuristics. This suggests us to study the large scale resonant system
\begin{align}
i\pt_t u_j+\Delta_x u_j=-\bg(\sum_{(j_1,j_2,j_3)\in\mathcal{RS}_j}u_{j_1}\bar u_{j_2}u_{j_3}\bg)u_j,\quad j\in\Z
\end{align}
on $\R_x^2$. By fundamental counting combinatorics, the large scale resonant system can be reformulated to
\begin{align}\label{nls}
i\pt_t u_j+\Delta_x u_j=-\bg(\sum_{i}|u_i|^2+\sum_{i\neq j}|u_i|^2\bg)u_j,\quad j\in\Z,
\end{align}
which will be the main model under consideration in the remaining part of the present paper.

Before we turn to the main results, we recall several conservation laws and symmetry invariance of the NLS which will be useful for the upcoming proofs. For the NLS \eqref{cnls}, we have following classical conservation laws:
\begin{alignat}{2}
\text{Mass}:&\quad \mM(U)&&=\|U\|^2_{L^2(\R^2\times\T)},\\
\text{Energy}:&\quad \mH(U)&&=\frac{1}{2}\|\nabla_{x,y} U\|^2_{L^2(\R^2\times\T)}-\frac14\|U\|_{L^4(\R^2\times\T)}^4,\\
\text{Momentum}:&\quad \mathcal{P}(U)&&=\mathrm{Im}\int_{\R^2\times\T}\overline{U}\nabla_{x,y}U\,dxdy.
\end{alignat}
For the large scale resonant system \eqref{nls}, the following conservation laws hold (see \cite{HaniPausader}):
\begin{alignat}{2}
\text{Mass}:&\quad M_0(\vu)&&=\|\vu\|^2_{\ell^2L_x^2}=\sum\nolimits_{j} \|\vu\|^2_{L^2(\R^2)},\\
\text{Weighted mass}:&\quad M_1(\vu)&&=\|\vu\|^2_{h^1L_x^2}=\sum\nolimits_{j} \la j\ra^2\|\vu\|^2_{L^2(\R^2)},\\
\text{Energy}:&\quad E(\vu)&&=\frac{1}{2}\|\nabla_x\vu\|^2_{\ell^2L_x^2}-\frac14\int_{\R^2}\bg(\sum_j|u_j|^2\bg)^2
+\sum_j\bg(|u_j|^2\sum_{i\neq j}|u_i|^2\bg)\,dx.
\end{alignat}
Moreover, by direct calculation it is also immediate that \eqref{cnls} and \eqref{nls} are invariant under the Galilean transformation
\begin{align}
U(t,x,y)&\mapsto e^{i\xi\cdot x}e^{-it|\xi|^2}U(t,x-2\xi t,y),\\
u(t,x)&\mapsto e^{i\xi\cdot x}e^{-it|\xi|^2}u(t,x-2\xi t)\label{vector}
\end{align}
for arbitrary $\xi\in\R^2$, where the Galilean transformation in \eqref{vector} is understood componentwise.

\subsection{Main results}

We begin with formulating the large data scattering result for the large scale resonant system \eqref{nls}. Following the idea in \cite{weinstein}, we define the Weinstein problem by
\begin{align}\label{weinstein problem}
\mathrm{C}_{\rm GN,rs}:=\inf_{\vu\in \ell^2 H_x^1}\frac{\|\vu\|^2_{\ell^2 L_x^2}\|\nabla\vu\|^2_{\ell^2 L_x^2}}{\int_{\R^2}(\sum_i|\vu_i|^2)^2+\sum_{j}(\sum_{i\neq j}|\vu_i|^2)|\vu_j|^2\,dx}
\end{align}
As revealed in \cite{weinstein}, the Weinstein problem \eqref{weinstein problem} is closely related to the Gagliardo-Nirenberg inequality and provides a sharp threshold for well-posedness problems of NLS of focusing type. Our first result gives a precise description of the constant $\mathrm{C}_{\rm GN,rs}$ in term of the 2D Euclidean ground state of the focusing cubic NLS.
\begin{proposition}[Large scale Gagliardo-Nirenberg inequality]\label{gagliardo nirenberg}
Define
\begin{align}\label{standard gn ineq}
\mathrm{C}_{\mathrm{GN},2d}:=\inf_{u\in H^1(\R^2)}\frac{\|u\|^2_{L^2(\R^2)}\|\nabla u\|^2_{L^2(\R^2)}}{\|u\|_{L^4(\R^2)}^4}.
\end{align}
Then $\mathrm{C}_{\rm GN,rs}=\frac{1}{2}\mathrm{C}_{\mathrm{GN},2d}$.
\end{proposition}

\begin{remark}\label{pohozaeve remark}
By Pohozaev's identity (see for instance \cite{lions1}) it is immediate that
$$\mathrm{C}_{\rm GN,rs}=\frac{1}{2}\mathrm{C}_{\mathrm{GN},2d}=\frac14\mM(Q_{2d}),$$
where $Q_{2d}$ is the unique positive, radial solution of
\begin{align}
-\Delta Q_{2d}+Q_{2d}=Q_{2d}^3\quad\text{on $\R^2$}.
\end{align}
\end{remark}
Proposition \ref{gagliardo nirenberg} motivates the following large data scattering result for \eqref{nls}:

\begin{theorem}[Large data scattering for the large scale resonant system]\label{thm large scale resonant}
Let $\vu_0\in h^1 L_x^2$ satisfy $M_0(\vu)<\frac12\mM(Q_{2d})$. Then a solution $\vu$ of \eqref{nls} with $\vu(0)=\vu_0$ is global and scattering in time, i.e. there exist $\phi^{\pm}\in h^1 L_x^2$ such that
\begin{align}
\lim_{t\to\pm\infty}\|u(t)-e^{it\Delta_x}\phi^{\pm}\|_{h^1 L_x^2}=0.
\end{align}
\end{theorem}

The proof follows the standard concentration compactness arguments from \cite{KenigMerle2006}. In order to exclude the minimal-blowup solution, we invoke the so-called long time Strichartz estimate to rule out the rapid cascade and quasi-soliton scenarios. Such long time Strichartz estimates were initiated by Dodson \cite{dodson1d,dodson2d,dodson3d,Dodson4dmassfocusing,BookDodson} for the study of Euclidean mass-critical NLS. The one we use in this paper is the vector-valued variant deduced in \cite{Yang_Zhao_2018} for the defocusing analogue of \eqref{nls}. We point out that ruling out the rapid cascade scenario is just a straightforward modification of the same arguments given in \cite{dodson2d,Yang_Zhao_2018}, by combining also Proposition \ref{gagliardo nirenberg}. However, the Planchon-Vega-type interaction Morawetz inequality \cite{2dinteraction} applied in \cite{Yang_Zhao_2018} can not be used for the focusing model to rule out the quasi-soliton scenario. Alternatively, we utilize the potentials constructed in \cite{Dodson4dmassfocusing} to achieve this goal.

We now turn our attention to the main model \eqref{cnls}. As usual, the focusing nature of \eqref{cnls} generally does not admit scattering for arbitrary initial data and a suitable variational analysis for formulating scattering thresholds will be necessary. The starting point of our variational analysis is the following scale-invariant (w.r.t. $x$-variable) Gagliardo-Nirenberg inequality of additive type.

\begin{proposition}[Gagliardo-Nirenberg inequality on $\R^2\times\T$]\label{prop gn cnls}
Let $\mathrm{C}_{\T}$ be the best constant of the inequality
\begin{align}
\|u\|^4_{L^4(\T)}\leq \mathrm{C}_{\T}\|u\|^{3}_{L^2(\T)}\|\nabla_y u\|_{L^2(\T)}
\end{align}
for functions $u\in H^1(\T)$ with $\int_{\T}u\,dy=0$. Let also
\begin{align}
\widehat{\rm G}_{\mathrm{GN},2d}:=\inf_{u\in H^1(\R^2)}\frac{\|u\|^2_{L^2(\R^2)}\|\nabla_x u\|^4_{L^2(\R^2)}}{\|u\|^6_{L^6(\R^2)}}.
\end{align}
Then there exists some $c\in(0,\mathrm{C}_{\T}^{\frac14}\widehat{\rm G}_{\mathrm{GN},2d}^{-\frac18}(1+(2\pi)^{-\frac12})^{\frac34}]$ such that for all $u\in H^1(\R^2\times\T)$ we have
\begin{align}
\|u\|_{L^4(\R^2\times\T)}&\leq
\|\nabla_x u\|^{\frac12}_{L^2(\R^2\times\T)}\bg((\pi\mM(Q_{2d}))^{-\frac14}\|u\|^{\frac12}_{L^2(\R^2\times\T)}
+c\|u\|^{\frac14}_{L^2(\R^2\times\T)}\|\nabla_y u\|^{\frac14}_{L^2(\R^2\times\T)}\bg)
\label{r2t1gn}.
\end{align}
\end{proposition}

\begin{remark}\label{sharp}
Here follow several comments on Proposition \ref{prop gn cnls}.
\begin{itemize}
\item[(i)] The existence of $\widehat{\rm G}_{\mathrm{GN},2d}$ follows from the classical Euclidean Gagliardo-Nirenberg inequality. The existence of $\mathrm{C}_{\T}$ will be shown in the proof of Proposition \ref{prop gn cnls}.

\item[(ii)]The constant $(\pi\mM(Q_{2d}))^{-\frac14}$ is sharp in the sense that there exists no non-negative number $\tilde{c}\geq0$ such that \eqref{r2t1gn} holds when $(\pi\mM(Q_{2d}))^{-\frac14}$ is replaced by a smaller number and $c$ is replaced by $\tilde{c}$. Indeed, we can simply take $u$ independent of $y\in\T$ and the second term in \eqref{r2t1gn} is equal to zero. Then replacing  $(\pi\mM(Q_{2d}))^{-\frac14}$ by any smaller number would lead to a contradiction to \eqref{standard gn ineq}. On the other hand, \eqref{r2t1gn} can not hold for $c=0$. To see this, we can simply insert $u(x,y)=Q_{2d}(x)\phi(y)$ into \eqref{r2t1gn} to obtain the contradiction $\|\phi\|_{L^4(\T)}\leq (2\pi)^{-\frac14}\|\phi\|_{L^2(\T)}$ for all $\phi\in H^1(\T)$. However, we do not know if the number $\mathrm{C}_{\T}^{\frac14}\widehat{\rm G}_{\mathrm{GN},2d}^{-\frac18}(1+(2\pi)^{-\frac12})^{\frac34}$ is optimal.
\end{itemize}
\end{remark}

In view of \eqref{r2t1gn} and Remark \ref{sharp}, we define
\begin{align}
c_*:=\inf\{c>0:\,\text{\eqref{r2t1gn} holds for $c$}\}>0.\label{defcdelta}
\end{align}
Having all the preliminaries we are able to formulate the large data scattering result for \eqref{cnls}:
\begin{theorem}[Large data scattering for focusing cubic NLS on $\R^2\times\T$]\label{main thm}
Define
\begin{align}
\Gamma(m):=c_*^{-8}m^{-1}(2^{\frac14}-(\pi\mM(Q_{2d}))^{-\frac14}m^{\frac14})^8\label{gamma def}
\end{align}
and for $f\in L^2(\R^2\times\T)$ define $\Gamma(f):=\Gamma(\mM(f))$. Let $U_0\in H^1(\R^2\times\T)$ satisfy
\begin{gather}
\mM(U_0)\in(0,\pi\mM(Q_{2d})),\label{threshold1}\\
\mH(U_0)<2^{-1}\Gamma(U_0),\label{threshold2}\\
\|\nabla_y U_0\|_{L^2(\R^2\times\T)}^2<
\Gamma(U_0).\label{threshold3}
\end{gather}
Then the solution $U$ of \eqref{cnls} with $U(0)=U_0$ is global and scattering in time, i.e. there exist $\Phi^\pm\in H^1(\R^2\times\T)$ such that
\begin{align}
\lim_{t\to\pm\infty}\|U(t)-e^{it\Delta_{x,y}}\Phi^\pm\|_{H^1(\R^2\times\T)}=0.
\end{align}
\end{theorem}

Due to the energy-subcritical nature of \eqref{cnls} we also have the following global well-posedness (but not necessarily scattering) result for \eqref{cnls}, with the even weaker condition $\mM(U_0)\in(0,2\pi\mM(Q_{2d}))$ in place of \eqref{threshold1}.

\begin{theorem}\label{thm gwp}
Let the conditions in Theorem \ref{main thm}, up to \eqref{threshold1}, be retained, and let \eqref{threshold1} be replaced by the weaker condition
\begin{align}\label{weaker thres}
\mM(U_0)\in(0,2\pi\mM(Q_{2d})).
\end{align}
Then a solution $U$ of \eqref{cnls} with $U(0)=U_0$ is global.
\end{theorem}

\begin{proof}
The proof is short, thus we already record it at the beginning of the paper. Since \eqref{cnls} is energy-subcritical, it is well-known (see for instance \cite{Cazenave2003}) that global well-posedness is equivalent to the statement that for all $n\in\N$ we have
\begin{align}
\sup_{t\in[-n,n]}\|\nabla_{x,y}U(t)\|_{L^2(\R^2\times\T)}<\infty,
\end{align}
which follows immediately from Lemma \ref{holmer variational} below.
\end{proof}

After this work had been completed the author became aware that Cheng et al. independently studied the same problem in \cite{similar_cubic}. We point out that both of the works proved Proposition \ref{gagliardo nirenberg} and Theorem \ref{thm large scale resonant} (Theorem 1.1 in \cite{similar_cubic}). Nevertheless, the remaining topics of both papers have different emphases and are not covered by each other, which makes our contributions independent. We illustrate at this point more precisely to clarify the situation:
\begin{itemize}
\item[(i)] The large data scattering result for \eqref{cnls} formulated in \cite{similar_cubic} is valid for solutions with initial data $U(0)$ of the form $U(0)=\ld^{-1}U_0(\ld^{-1}x,y)$ for all $\ld\geq \ld(U_0)\gg 1$ (\cite[Thm. 1.1]{similar_cubic}). This result is in fact a direct consequence of Lemma \ref{cnls lem large scale proxy} (which essentially shares the same proof of \cite[Lem. 3.11]{CubicR2T1Scattering}) and only requires the initial data to be lying below ground state threshold (i.e. $\mM(U_0)<\pi \mM(Q_{2d})$), but at the same time also demands the initial data to take the somehow restricted form $\ld^{-1}U_0(\ld^{-1}x,y)$, where a quantitative description of the lower bound number $\ld(U_0)$ is not available. In Theorem \ref{main thm}, we formulated the thresholds with the additional (quantitative) constraints \eqref{threshold2} and \eqref{threshold3} but imposing no specific forms on the initial data. This scheme follows the same fashion as in the classical works \cite{KenigMerle2006,HolmerRadial,Dodson4dmassfocusing}. Moreover, in Theorem \ref{thm gwp} we also proved global well-posedness results beyond the threshold $\pi\mM(Q_{2d})$, which was not involved in \cite{similar_cubic}.

\item[(ii)] It was proved in \cite{similar_cubic} that the threshold in Proposition \ref{gagliardo nirenberg} is sharp in the sense that for any number $m>2^{-1}\mM(Q_{2d})$, finite time blow-up solutions $u$ of \eqref{nls} with $M_0(u)=m$ and $E(u)<0$ exist. The proof makes use of the classical Glassey's virial arguments \cite{Glassey1977}.

\item[(iii)] The authors of \cite{similar_cubic} also considered the $N$-coupled variant NLS-system of \eqref{nls} (namely only the components with (absolute) indices $\leq N$ are non-vanishing). This system arises as non-relativistic limit of the complex-valued cubic focusing nonlinear Klein-Gordon equation in $\R^2$, see for instance \cite{MasmoudiNakanishi2002}. Particularly, the authors showed the existence of ground states of the stationary $N$-coupled system (in general, however, we conjecture that there exists no optimizer for the infinite dimensional Weinstein problem \eqref{weinstein problem}). Moreover, the authors of \cite{similar_cubic} formulated a similar scattering threshold as the one given by Theorem \ref{thm large scale resonant} for the $N$-coupled system, which is strictly larger than $2^{-1}\mM(Q_{2d})$ for finite $N$.
\end{itemize}
The above mentioned differences hence reveal the independence of both papers, and we decide to keep the same contents of this paper in the following as before, without further modification according to \cite{similar_cubic}.

\subsection{Notation and definitions}
We use the notation $A\lesssim B$ whenever there exists some positive constant $C$ such that $A\leq CB$. Similarly we define $A\gtrsim B$ and we use $A\sim B$ when $A\lesssim B\lesssim A$. For simplicity, we ignore in most cases the dependence of the function spaces on their underlying domains and hide this dependence in their indices. For example $\ell^2 L_x^2=\ell^2(\Z,L^2(\R^2))$, $H_{x,y}^1= H^1(\R^2\times \T)$
and so on. However, when the space is involved with time, we still display the underlying time domain such as $\diag(I)$, $L_t^\infty L_x^2(\R)$ etc. The space $h^1 \dot{H}_x^s$ is defined through the norm
\begin{align*}
\|f\|^2_{h^1 \dot{H}_x^s}:=\sum_{j}\la j \ra^2\|f_j\|^2_{\dot{H}_x^s}
\end{align*}
for $f:\Z\to \dot{H}_x^s$. We denote by $g_{\xi_0,x_0,\ld_0}$ the $L_x^2$-symmetry transformation defined by
\begin{align*}
g_{\xi_0,x_0,\ld_0}f(x):=\ld_0^{-1}e^{i\xi_0\cdot x}f(\ld_0^{-1}(x-x_0))
\end{align*}
for $(\xi_0,x_0,\ld_0)\in\R^2\times\R^2\times(0,\infty)$. We define the Fourier transformation of a function $f$ w.r.t. $x\in\R^2$ or $y\in\T$ by
\begin{align*}
\hat{f}_y(x,k)&=\mathcal{F}_yf(x,k):=(2\pi)^{-\frac{1}{2}}\int_{\T}f(x,y)e^{-iky}\,dy,\\
\hat{f}_x(\xi,y)&=\mathcal{F}_xf(\xi,y):=(2\pi)^{-1}\int_{\R^2}f(x,y)e^{-i\xi\cdot x}\,dx,\\
\hat{f}_{\xi,k}(\xi)&=\mathcal{F}_{x,y}f(\xi,k):=(2\pi)^{-\frac{3}{2}}\int_{\R^2\times\T}f(x,y)e^{-i(\xi\cdot x+ky)}\,dxdy.
\end{align*}
Let $\phi\in C^\infty_c(\R^2)$ be a fixed radial, non-negative and radially decreasing function such that $\psi(x)=1$ if $|x|\leq 1$ and $\psi(x)=0$ for $|x|\geq \frac{11}{10}$. Then for $N>0$, we define the Littlewood-Paley projectors w.r.t $x$-variable by
\begin{align*}
P_{\leq N} f(x)&=\mathcal{F}_x^{-1}\bg(\phi\bg(\frac{\xi}{N}\bg)\hat{f}(\xi)\bg)(x),\\
P_{N} f(x)&=\mathcal{F}_x^{-1}\bg(\bg(\phi\bg(\frac{\xi}{N}\bg)-\phi\bg(\frac{2\xi}{N}\bg)\bg)\hat{f}(\xi)\bg)(x),\\
P_{> N} f(x)&=\mathcal{F}_x^{-1}\bg(\bg(1-\phi\bg(\frac{\xi}{N}\bg)\bg)\hat{f}(\xi)\bg)(x).
\end{align*}
We also record the following well-known Bernstein inequalities which will be frequently used throughout the paper: For all $s\geq 0$ and $1\leq p\leq\infty$ we have
\begin{align*}
\|P_{> N}f\|_{L^p}&\lesssim N^{-s}\||\nabla|^s P_{> N}f\|_{L^p},\\
\||\nabla|^s P_{\leq N}f\|_{L^p}&\lesssim N^{s}\| P_{\leq N}f\|_{L^p},\\
\||\nabla|^{\pm s} P_{ N}f\|_{L^p}&\sim N^{\pm s}\| P_{ N}f\|_{L^p},\\
\|P_{\leq N}f\|_{L^q}&\lesssim N^{\frac{2}{p}-\frac{2}{q}}\|P_{\leq N}f\|_{L^p},\\
\|P_{N}f\|_{L^q}&\lesssim N^{\frac{2}{p}-\frac{2}{q}}\|P_{N}f\|_{L^p}.
\end{align*}

Next we introduce the concept of an \textit{admissible} pair on $\R^d$. A pair $(q,r)$ is said to be $\dot{H}^s$-admissible if $q,r\in[2,\infty]$, $s\in[0,\frac d2)$, $\frac{2}{q}+\frac{d}{r}=\frac{d}{2}-s$ and $(q,d)\neq(2,2)$. For any $L^2$-admissible pairs $(q_1,r_1)$ and $(q_2,r_2)$ we have the following Strichartz estimate: if $u$ is a solution of
\begin{align*}
i\pt_t u+\Delta_x u=F
\end{align*}
in $I\subset\R$ with $t_0\in I$ and $u(t_0)=u_0$, then
\begin{align*}
\|u\|_{L_t^q L_x^r(I)}\lesssim \|u_0\|_{L_x^2}+\|F\|_{L_t^{q_2'} L_x^{r_2'}(I)},
\end{align*}
where $(q_2',r_2')$ is the H\"older conjugate of $(q_2,r_2)$. For a proof, we refer to \cite{EndpointStrichartz,Cazenave2003}. Combining with Minkowski's inequality, for a vector $u=(u_j)_{j\in\Z}$ satisfying
\begin{align*}
i\pt_t u_j+\Delta_x u_j=F_j\quad\forall j\in\Z
\end{align*}
we also have the Strichartz estimate
\begin{align*}
\|u\|_{L_t^q L_x^r\ell^2(I)}\lesssim \|u_0\|_{L_x^2\ell^2}+\|F\|_{L_t^{q_2'} L_x^{r_2'}\ell^2(I)}.
\end{align*}
For $d=2$, we define the spaces $S_0,S_1$ by
\begin{align*}
S_0:=L_t^\infty L_x^2\cap L^{2^{+}}_t L_x^{\infty-},\qquad
S_1:=L_t^\infty \dot{H}^1\cap L^{2^{+}}_t \dot{W}_x^{1,\infty-},
\end{align*}
where $(2^{+},\infty^-)$ is an $L^2$-admissible pair with some sufficiently small $2^+\in(2,\infty)$. In the following, an admissible pair is always referred to as an $L^2$-admissible pair if not otherwise specified.

Finally, we denote by $F(\vu_m)=-(\sum_{i}|u_i|^2+\sum_{i\neq j}|u_i|^2)u_m$ the nonlinear potential of \eqref{nls} for the component $\vu_m$. When the vector $w$ is given by $w=(w_m)_m=(P\vu_m)_m$, where $P$ is some frequency projector, then we similarly define $F(w_m)$ as $F(w_m)=-(\sum_{i}|w_i|^2+\sum_{i\neq j}|w_i|^2)w_m$.

\section{Scattering for the large scale resonant system}\label{sec:Existence critical element resonant}
\subsection{Proof of Proposition \ref{gagliardo nirenberg}}
We begin with the proof of Proposition \ref{gagliardo nirenberg}.
\begin{proof}[Proof of Proposition \ref{gagliardo nirenberg}]
Recall that $Q_{2d}$ is the unique radially symmetric and positive solution of
\begin{align}
-\Delta Q_{2d}+Q_{2d}=Q_{2d}^3\quad\text{on $\R^2$}.
\end{align}
It is well-known by Pohozaev's identity (see \cite{lions1}) that
\begin{align}
\mathrm{C}_{\mathrm{GN},2d}=\frac{\|Q_{2d}\|^2_{L_x^2}\|\nabla Q_{2d}\|^2_{L_x^2}}{\|Q_{2d}\|_{L_x^4}^4}.
\end{align}
We now set $\vu^n=Q_{2d}\sum_{|i|\leq n}e_i$. Then direct calculation yields
\begin{align}
\|\vu^n\|^2_{\ell^2 L_x^2}\|\nabla\vu^n\|^2_{\ell^2 L_x^2}&=(4n^2+4n+1)\|Q_{2d}\|_{L_x^2}^2\|\nabla Q_{2d}\|_{L_x^2}^2,\\
\int_{\R^2}\bg(\sum_i|\vu^n_i|^2\bg)^2+\sum_{j}\bg(\sum_{i\neq j}|\vu^n_i|^2\bg)|\vu_j|^2\,dx
&=(8n^2+6n+1)\|Q_{2d}\|_{L_x^4}^4.
\end{align}
Hence $\mathrm{C}_{\rm GN,rs}\leq \frac{4n^2+4n+1}{8n^2+6n+1}\mathrm{C}_{\mathrm{GN},2d}$. Sending $n\to\infty$ we obtain $\mathrm{C}_{\rm GN,rs}\leq \frac{1}{2}\mathrm{C}_{\mathrm{GN},2d}$. Using Minkowski, \eqref{standard gn ineq} and H\"older we obtain
\begin{align}
\int_{\R^2}\bg(\sum_{j}|\vu_j|^2\bg)^2\,dx\leq \bg(\sum_j\bg(\int_{\R^2}|\vu_j|^4\,dx\bg)^{\frac{1}{2}}\bg)^2
\leq \bg(\sum_j\mathrm{C}_{\mathrm{GN},2d}^{-\frac{1}{2}}\|\vu_j\|_2\|\nabla\vu_j\|_2\bg)^2
\leq \mathrm{C}_{\mathrm{GN},2d}^{-1}\|\vu\|^2_{\ell^2 L_x^2}\|\nabla\vu\|^2_{\ell^2 L_x^2}.
\end{align}
The desired inequality then follows from the rough estimate
\begin{align}
\int_{\R^2}\bg(\sum_i|\vu_i|^2\bg)^2+\sum_{j}\bg(\sum_{i\neq j}|\vu_i|^2\bg)|\vu_j|^2\,dx
\leq 2\int_{\R^2}\bg(\sum_{j}|\vu_j|^2\bg)^2\,dx.
\end{align}
\end{proof}

\subsection{Existence of a minimal blow-up solution for the large scale resonant system}
Next, we establish a result concerning the existence of a minimal blow-up solution of \eqref{nls} when assuming that Theorem \ref{thm large scale resonant} does not hold. The proof is almost identical to \cite[Thm. 3.3]{Yang_Zhao_2018}, where we only need to add the additional mass constraint to the inductive hypothesis, thus we omit the details here.

\begin{theorem}[Existence of a minimal blow-up solution]\label{thm existence critical element}
Suppose that Theorem \ref{thm large scale resonant} does not hold. Then there exists a solution $\vu_c$ of \eqref{nls} such that $M_0(\vu_c)<\frac{1}{2}\mM(Q_{2d})$ and
\begin{align}\label{infinite scattering norm main thm}
\|\vu_c\|_{\diag\ell^2((\inf I_{\max},0])}=\|\vu_c\|_{\diag\ell^2([0,\sup I_{\max}))}=\infty.
\end{align}
Moreover, the set $\{\vu_c(t):t\in I_{\max}\}$ is precompact in $h^1 L_x^2$ modulo $L_x^2$-symmetries.
\end{theorem}

\subsection{Properties of the almost periodic solution}
In this subsection we collect some useful properties of the minimal blow-up solution $\vu_c$.

\begin{lemma}[Arzela-Ascoli characterization of $h^1L^2$-compactness, \cite{Yang_Zhao_2018}]\label{arzela ascoli}
Let $\vu$ be an almost periodic solution of \eqref{nls}. Then there exist functions $x:I\to \R^2$, $\xi:I\to\R^2$, $C:(0,\infty)\to (0,\infty)$ and $N:I\to (0,\infty)$ such that for any $\eta>0$ and any $t\in I$ we have
\begin{align}\label{almost periodic smallness}
\sum_j\la j\ra^2\int_{|x-x(t)|\geq \frac{C(\eta)}{N(t)}}|\vu_j(t,x)|^2\,dx
+\sum_j\la j\ra^2\int_{|\xi-\xi(t)|\geq C(\eta)N(t)}|\hat{\vu}_j(t,\xi)|^2\,d\xi<\eta^2.
\end{align}
\end{lemma}

\begin{lemma}[Normalisation of the symmetry functions, \cite{killip_tao_visan_2d_mass_critical}]
We may additionally assume that the minimal blow-up solution $\vu_c$ deduced from Theorem \ref{thm existence critical element} satisfies the following:
\begin{itemize}
\item[(i)] The maximal interval $I$ contains at least $[0,\infty)$.
\item[(ii)]We have $\|\vu_c\|_{\diag\ell^2([0,\infty))}=\infty$.
\item[(iii)]The functions $x,\xi,N$ can be chosen such that $x(0)=\xi(0)=0$, $N(0)=1$ and $N(t)\leq 1$ for all $t\in[0,\infty)$.
\end{itemize}
\end{lemma}

\begin{lemma}[Local constancy of $N(t)$, \cite{killip_tao_visan_2d_mass_critical}]\label{lem local constancy}
Let $\vu:I\times\R^d\to\C$ be a non-zero maximal-lifespan solution of \eqref{nls} that is almost periodic modulo symmetries and has the frequency scale function $N$. Then there exists a small $\delta=\delta(u)>0$ such that for every $t_0\in I$ we have
\begin{align}
[t_0-\delta N(t_0)^{-2},t_0+\delta N(t_0)^{-2}]\subset I.
\end{align}
Moreover, $N(t)\sim_u N(t_0)$ whenever $|t-t_0|\leq \delta N(t_0)^{-2}$.
\end{lemma}

\begin{lemma}[Spacetime bound, \cite{killip_tao_visan_2d_mass_critical}]\label{local constancy lemma}
Let $\vu:I\times\R^d\to\C$ be a non-zero maximal-lifespan solution of \eqref{nls} that is almost periodic modulo symmetries and has the frequency scale function $N$. Let $J$ be any subinterval of $I$. Then
\begin{align}\label{local constancy -1}
\int_J N^2(t)\,dt\lesssim \|\vu\|^4_{\diag \ell^2(J)}\lesssim 1+\int_J N^2(t)\,dt.
\end{align}
\end{lemma}

The following result is an immediate consequence of Lemma \ref{lem local constancy} and Lemma \ref{local constancy lemma}.
\begin{corollary}
Suppose that for some interval $J$ we have $\|\vu\|_{\diag \ell^2 (J)}=1$. Then $\sup_{t\in J}N(J)\lesssim_{\vu} \inf_{t\in J}N(t)$. Moreover, if a time interval $J$ can be partitioned into consecutive intervals $J=\cup J_{\ell}$ with $\|\vu\|_{\diag\ell^2(J_{\ell})}=1$, then $\sum_{J_{\ell}}\sup_{t\in J_\ell}N(t)\sim_{\vu}\int_J N(t)^3\,dt$.
\end{corollary}

\subsection{Impossibility of solutions of rapid cascade type}
In this section we rule out the rapid cascade scenario, i.e. the case $\int_0^\infty N(t)^3\,dt<\infty$. We firstly state the following lemma proved in \cite{dodson2d,Yang_Zhao_2018}, which confirms the higher regularity of the minimal blow-up solution in the rapid cascade case.
\begin{lemma}[\cite{dodson2d,Yang_Zhao_2018}]\label{lem higher reg}
Let $\vu$ be the almost periodic solution of \eqref{nls} given by Theorem \ref{thm existence critical element}. If $\int_0^\infty N(t)^3\,dt=K<\infty$, then $\|\vu\|_{L_t^\infty \ell^2 \dot{H}_x^2(0,\infty)}<\infty$.
\end{lemma}

\begin{lemma}[Impossibility of almost periodic solution of rapid cascade type]\label{impos rapid cascade}
Let $\vu$ be the almost periodic solution of \eqref{nls} given by Theorem \ref{thm existence critical element}. If $\int_0^\infty N(t)^3\,dt=K<\infty$, then $\vu\equiv 0$.
\end{lemma}

\begin{proof}
From Lemma \ref{lem local constancy} it follows $|N'(t)|\lesssim N(t)^3$, which in turn implies $N(t)\to 0$ as $t\to\infty$. Combining with the fact that $\vu$ is almost periodic modulo $L_x^2$-symmetries we infer that for any $\eta>0$
\begin{align}
\lim_{t\to\infty}\|P_{\xi(t),\leq C(\eta)N(t)}\vu(t)\|_{\ell^2 L_x^2}=0,
\end{align}
which combining with interpolation and the fact that $\vu\in L_t^\infty\ell^2 \dot{H}^2_x(0,\infty)$ deduced from Lemma \ref{lem higher reg}
\begin{align}
\lim_{t\to\infty}\|P_{\xi(t),\leq C(\eta)N(t)}\vu(t)\|_{\ell^2 \dot{H}_x^1}=0.
\end{align}
Now using interpolation and \eqref{almost periodic smallness} we obtain
\begin{align}
\|P_{\xi(t),\geq C(\eta)N(t)}\vu(t)\|_{\ell^2 \dot{H}_x^1}\lesssim \eta^{\frac12}
\end{align}
for all $t\in(0,\infty)$. Since $\eta$ is chosen arbitrarily, we conclude that
\begin{align}
\lim_{t\to\infty}\|\vu(t)\|_{\ell^2 \dot{H}_x^1}=0.
\end{align}
By Proposition \ref{gagliardo nirenberg}, conservation of mass and energy
\begin{align}
E(\vu(0))=\lim_{t\to\infty}E(\vu(t))\lesssim \lim_{t\to\infty}\|\vu(t)\|^2_{\ell^2 \dot{H}_x^1}=0.
\end{align}
But using Proposition \ref{gagliardo nirenberg} again, we infer that
\begin{align}
\|\nabla \vu(0,x)\|^2_{\ell^2 L_x^2}\leq 2\bg(1-\frac{2M_0(\vu)}{\mM(Q_{2d})}\bg)^{-1}E(\vu(0))=0,
\end{align}
which in turn implies $\vu=0$. This completes the proof.
\end{proof}

\subsection{Impossibility of solutions of quasi-soliton type}
In this section we rule out the quasi-soliton scenario, i.e. the case $\int_0^\infty N(t)^3\,dt=\infty$. First, we denote by $\vare_3$ the small constant related to the long time Strichartz estimate, which is the same constant defined in \cite[Sec. 5]{Yang_Zhao_2018}. The construction of the long time Strichartz estimate is however very cumbersome and will not be directly applied for the upcoming proofs, thus we omit the details. Let $T>0$ and define $K=\int_0^T N(t)^3\,dt$. Define $w=(w_m)_m=P_{\leq \vare_3^{-1}K}\vu$. Then
\begin{align}
i\pt_t w_m+\Delta w_m=F(w_m)+(P_{\leq \vare_3^{-1}K}F(\vu_m)-F(w_m))=:F(w_m)+N_m.
\end{align}
Let $a(t,x)=(a_j(t,x))_{j=1,2}:I\times\R^2\to\R^2$ be some to be determined potentials. We define the frequency localized interaction Morawetz action $M(t)$ by
\begin{align}
M(t)&=2\sum_{n,m}\int\int|w_n(t,y)|^2a(t,x-y)\mathrm{Im}(\bar w_m\nabla w_m)(t,x)\,dxdy.
\end{align}
Integration by parts yields
\begin{align}
&\,\frac{d}{dt}M(t)\label{morawetz small 1}\\
=&\,4\int\int (\sum_n|w_n|^2)(t,y)\pt_ka_j(t,x-y)\mathrm{Re}(\sum_m\pt_j\bar w_m \pt_k w_m)(t,x)\,dxdy\label{remaining 1}\\
-&\,4\int\int\mathrm{Im}(\sum_n\bar w_n \pt_k w_n)(t,y)\pt_k a_j(t,x-y)\mathrm{Im}(\sum_m\bar w_m \pt_j w_m)(t,x)\,dxdy\label{remaining 2}\\
+&\,\int\int(\sum_n|w_n|^2)(t,y)a_j(t,x-y)\pt_j\pt_k^2(\sum_m|w_m|^2)(t,x)\,dxdy\label{remaining 3}\\
+&\,\int\int (\sum_n|w_n|^2)(t,y) \pt_j a_j(t,x-y)(\sum_m F(w_m)\bar w_m)(t,x)\,dxdy\label{remaining 4}\\
+&\,2\int\int(\sum_n|w_n|^2)(t,y)\pt_ta_j(t,x-y)\mathrm{Im}(\sum_m\bar w_m\pt_j w_m)(t,x)\,dxdy\label{remaining 5}\\
+&\,4\int\int\mathrm{Im}(\sum_n\bar w_n N_n)(t,y)a_j(t,x-y)\mathrm{Im}(\sum_m\bar w_m \pt_j w_m)(t,x)\,dxdy\label{morawetz small 2}\\
+&\,2\int\int (\sum_n|w_n|^2)(t,y) a_j(t,x-y)\mathrm{Re}(\sum_m N_m\pt_j\bar w_m)(t,x)\,dxdy\label{morawetz small 3}\\
-&\,2\int\int (\sum_n|w_n|^2)(t,y) a_j(t,x-y)\mathrm{Re}(\sum_m w_m\pt_j\bar N_m)(t,x)\,dxdy\label{morawetz small 4}.
\end{align}
The following lemma shows that \eqref{morawetz small 1}, \eqref{morawetz small 2}, \eqref{morawetz small 3} and \eqref{morawetz small 4} are ignorable for suitable potentials $(a_j)_j$ and sufficiently large $K=\int_0^T N(t)^3\,dt$.
\begin{lemma}[\cite{dodson2d,Yang_Zhao_2018}]\label{lem small K}
Let $R>0$. Assume that the potentials $(a_j)_j$ are real and satisfy
\begin{align}
|a_j(t,x)|&\leq R,\label{require1}\\
|\nabla a_j(t,x)|&\leq R|x|^{-1},\label{gradient}\\
a_j(t,x)&=-a_j(t,-x)\label{require3}
\end{align}
for all $j=1,2$, $t\in[0,T]$ and $x\in\R^2$. Then
\begin{align}
\sup_{t\in [0,T]}|M(t)|+\int_0^T\eqref{morawetz small 2}+\eqref{morawetz small 3}+\eqref{morawetz small 4}\,dt\lesssim Ro(K),
\end{align}
where $o(K)$ is some quantity such that $o(K)/K\to 0$ as $K\to\infty$.
\end{lemma}

We therefore from now on focus on the remaining terms in the interaction Morawetz action. Before we finally exclude the quasi-soliton scenario, we still need the following asymptotic smallness lemma.

\begin{lemma}\label{lem before last}
Let $J$ be a time interval such that $\|\vu\|_{\diag\ell^2(J)}\lesssim 1$. Then
\begin{align}
\|\mathds{1}_{\{|x-x(t)|\geq RN(t)^{-1}\}}\vu\|_{L_{t,x}^4\ell^2(J)}+
\|P_{|\xi-\xi(t)|\geq RN(t)}\vu\|_{L_{t,x}^4\ell^2(J)}=o_R(1)
\end{align}
as $R\to\infty$.
\end{lemma}

\begin{proof}
By Duhamel's formula, Strichartz estimate and conservation of mass we know that $\|\vu\|_{L_t^p L_x^q \ell^2(J)}
\lesssim 1$ for arbitrary admissible $(p,q)$. Then the desired claim follows from \eqref{arzela ascoli} and interpolation between $L_t^\infty L_x^2 \ell^2$ and $L_t^p L_x^q \ell^2$ for some admissible $(p,q)$ with $p\in(2,4)$.
\end{proof}

Having all the preliminaries we are in the position to rule out the quasi-soliton scenario.

\begin{lemma}[Impossibility of almost periodic solution of quasi-soliton type]\label{lem quasi soliton}
Let $\vu$ be the almost periodic solution of \eqref{nls} given by Theorem \ref{thm existence critical element}. If $\int_0^\infty N(t)^3\,dt=\infty$, then $\vu\equiv 0$.
\end{lemma}

\begin{proof}
First, we construct the potentials $a_j(t,x)$ as follows: For $R>0$, let $\varphi\in C_c^\infty(\R^2;[0,1])$ be a radial and decreasing function such that $\varphi\leq 1$, $\varphi(z)\equiv 1$ on $B_{R-\sqrt{R}}(0)$ and $\mathrm{supp}\,\varphi\subset B_{R}(0)$\footnote{With slight abuse of notation we identify $\varphi:\R^2\to[0,\infty)$ with the same function $\varphi:[0,\infty)\to[0,\infty)$. The same convention is made for other radial functions.}. Next, define
\begin{align}
\phi(x)=\frac{1}{|B_R(0)|}\int \varphi(x-z)\varphi(z)\,dz.
\end{align}
Particularly, $\phi$ is non-negative and radial, supported on $B_{2R}(0)$ and $\sup_{x\in\R^2}\phi(x)\leq 4$. Moreover, we have
\begin{align}
\phi(x-y)=\frac{1}{|B_R(0)|}\int \varphi(x-z)\varphi(y-z)\,dz.
\end{align}
By \cite[Lem. 6.6]{decreasing}, the function $\phi$ is also decreasing. Finally, for $M>0$ define
\begin{align}
\psi_M(x)=\psi_M(|x|)=\frac{1}{|x|}\int_0^{|x|}\phi(\frac{s}{M})\,ds.
\end{align}
As immediate consequences, we have
\begin{gather}
r\psi_M'(r)=\phi(\frac{r}{M})-\psi_M(r)\leq 0,\label{equality}\\
x_j\psi_M(|x|)\leq 8RM,\\
\nabla(x_j\psi_M(|x|)\leq \frac{24RM}{|x|}.
\end{gather}
We now define
\begin{align}
a_j(t,x)=N(t)x_j\psi_{RN(t)^{-1}}(x).
\end{align}
Then the assumptions of $a_j$ in Lemma \ref{lem small K} are satisfied (with $R$ replaced by $24R^2$) and in view of Lemma \ref{lem small K}, the proof of Lemma \ref{lem quasi soliton} follows as long as we can prove
\begin{align}
\int_0^T\eqref{remaining 1}+\eqref{remaining 2}+\eqref{remaining 3}+\eqref{remaining 4}+\eqref{remaining 5}\,dt\gtrsim K
\end{align}
for sufficiently large $K$. Let us first take \eqref{remaining 3}. Straightforward calculation results in
\begin{align}
&\,\eqref{remaining 3}\nonumber\\
=&\,-\sum_{m,n}\int\int N(t)\Delta(\phi(R^{-1}N(t)|x-y|)+\psi_{RN(t)^{-1}}(x-y))
|w_m(t,x)|^2|w_n(t,y)|^2\,dxdy.
\end{align}
Moreover, by product rule and chain rule
\begin{align}
\Delta(\phi(x)+\psi_{1}(x))=(\phi''(|x|)+\psi_1''(|x|))+|x|^{-1}(\phi'(|x|)+\psi_1'(|x|)).
\end{align}
First recall that $\sup_{s\geq 0}|\varphi(s)|\lesssim 1$,  $\sup_{s\geq 0}|\varphi'(s)|\lesssim R^{-\frac12}$ and $\varphi'(s)$ is supported on $R-\sqrt{R}\leq|s|\leq R$ . By definition of $\phi$, we obtain
\begin{align}
\phi'(r)=\frac{d}{dr}(\phi(re_1))=|B_R(0)|^{-1}\int\varphi'(|re_1-z|)\frac{r-z_1}{|re_1-z|}\varphi(z)\,dt\lesssim R^{-\frac32}.
\end{align}
In the same manner, we deduce $\phi''(r)\lesssim R^{-2}$.
Now we observe that
\begin{align}
\phi'(0)=-|B_R(0)|^{-1}\int\frac{z_1}{|z|}\varphi'(|z|)\varphi(z)\,dt=0,
\end{align}
thus $\phi'(r)=\int_0^r \phi''(r)\,dr\lesssim R^{-2}r$. Next, using \eqref{equality}
\begin{align}
r^{-1}\psi_1'(r)
= r^{-3}\int_0^r\int_s^r\phi'(t)\,dtds\lesssim R^{-2}.\label{galilean4}
\end{align}
Similarly we infer that $\psi_1''(r)\lesssim R^{-2}$, which in turn implies
\begin{align}
\Delta(\phi(x)+\psi_1(x))\lesssim R^{-2}.
\end{align}
Therefore by conservation of mass
\begin{align}
\int_0^T\eqref{remaining 3}\,dt\gtrsim -\int_0^T\frac{N(t)^3}{R^4}\|w(t)\|^4_{\ell^2 L_x^2}\,dt= o_R(1)K.
\end{align}
Let us now consider \eqref{remaining 1}, \eqref{remaining 2} and \eqref{remaining 4}. Define the radial and angular derivatives $\nabla_{r,y}$ and $\ang_{y}$ centered at a point $y\in\R^2$ by
\begin{align}
\nabla_{r,y,j}=\frac{(x_j-y_j)\pt_j}{|x-y|},\quad \ang_{y,j}=\pt_j-\nabla_{r,y,j}.
\end{align}
Then it is straightforward by direct calculation to verify that for a function $f$ we have the decomposition
\begin{align}
|\nabla f|^2=|\nabla_{r,y}f|^2+|\ang_y f|^2.\label{decomp}
\end{align}
Combining with \eqref{equality} and Cauchy-Schwarz we infer that
\begin{align}
&\,\eqref{remaining 1}+\eqref{remaining 2}\nonumber\\
=&\,4\sum_{m,n}\int\int N(t)\phi(R^{-1}N(t)|x-y|)|\nabla w_m(t,x)|^2|w_n(t,y)|^2\,dxdy\nonumber\\
-&\,4\sum_{m,n}\int\int N(t)\phi(R^{-1}N(t)|x-y|)\mathrm{Im}(\bar w_n\pt_k w_n)(t,y)
\mathrm{Im}(\bar w_m\pt_k w_m)(t,x)\,dxdy\nonumber\\
+&\,4\sum_{m,n}\int\int N(t)(\phi(R^{-1}N(t)|x-y|)-\psi_{R/N(t)}(x-y))|w_n(t,y)|^2
|\ang_{y}w_m(t,x)|^2\,dxdy\nonumber\\
-&\,4\sum_{m,n}\int\int N(t)(\phi(R^{-1}N(t)|x-y|)-\psi_{R/N(t)}(x-y))
\mathrm{Im}(\bar w_n \ang_x w_n)(t,y)\mathrm{Im}(\bar w_m \ang_y w_m)(t,x)\,dxdy\nonumber\\
\geq &\,4\sum_{m,n}\int\int N(t)\phi(R^{-1}N(t)|x-y|)|\nabla w_m(t,x)|^2|w_n(t,y)|^2\,dxdy\nonumber\\
-&\,4\sum_{m,n}\int\int N(t)\phi(R^{-1}N(t)|x-y|)\mathrm{Im}(\bar w_n\pt_k w_n)(t,y)
\mathrm{Im}(\bar w_m\pt_k w_m)(t,x)\,dxdy\nonumber\\
=&\,4|B_R(0)|^{-1}\sum_{m,n}N(t)\int\int\int
\bg(\varphi(R^{-1}N(t)x-z)|\nabla w_m(t,x)|^2\bg)
\bg(\varphi(R^{-1}N(t)y-z)|w_n(t,y)|^2\bg)\,dzdxdy\qquad\label{galilean1}\\
-&\,4 |B_R(0)|^{-1}N(t)\sum_{m,n}\int\int\int
\bg(\varphi(R^{-1}N(t)x-z)\mathrm{Im}(\bar w_m\pt_k w_m)(t,x)\bg)\nonumber\\
&\qquad\qquad\qquad
\qquad\qquad\qquad
\times\bg(\varphi(R^{-1}N(t)y-z)\mathrm{Im}(\bar w_n\pt_k w_n)(t,y)\bg)\,dzdxdy.\label{galilean2}
\end{align}
The sum of \eqref{galilean1} and \eqref{galilean2} is invariant under the Galilean transformation $w\mapsto e^{-ix\cdot\beta_0}w$ for arbitrary $\beta_0\in\R^2$ (which can be easily checked by carefully expanding the terms in \eqref{galilean2} and using product and chain rules, we omit the straightforward but tedious details here). Thus we choose $t\mapsto \beta(t)$ such that
$$\sum_m\int\varphi(R^{-1}N(t)x-z)\mathrm{Im}(\overline{e^{-ix\cdot\beta(t)}w_m}\nabla( e^{-ix\cdot\beta(t)}w_m))(t,x)=0$$
and we are left with the term
\begin{align}
&\,4|B_R(0)|^{-1}\sum_{m,n}N(t)\int\int\int
\bg(\varphi(R^{-1}N(t)x-z)|\nabla (e^{-ix\cdot\beta(t)}w_m(t,x))|^2\bg)\nonumber\\
&\qquad\qquad\qquad\times\bg(\varphi(R^{-1}N(t)y-z)|w_n(t,y)|^2\bg)\,dzdxdy.\label{galilean3}
\end{align}
The existence of such a function $t\mapsto \beta(t)$ is guaranteed by \eqref{arzela ascoli} and the fact that $\vu\neq 0$. On the other hand, \eqref{equality}, \eqref{galilean4} and the fact that $F(w_m)\bar w_m\leq 0$ result in
\begin{align}
&\,\eqref{remaining 4}\nonumber\\
=&\,\sum_{m,n}\int\int N(t)(2\psi_{RN(t)^{-1}}(x-y)+|x|\psi'_{RN(t)^{-1}}(x-y))(F(w_m)\bar w_m)(t,x)|w_n(t,y)|^2\,dxdy\nonumber\\
\geq &\,\sum_{m,n}\int\int 2N(t)\psi_{RN(t)^{-1}}(x-y)(F(w_m)\bar w_m)(t,x)|w_n(t,y)|^2\,dxdy\nonumber\\
=&\,|B_R(0)|^{-1}\sum_{m,n}\int\int\int 2 N(t)\bg(\varphi(R^{-1}N(t)x-z)(F(w_m)\bar w_m)(t,x)\bg)\nonumber\\
&\qquad\qquad\qquad\qquad\qquad\qquad\qquad\qquad\qquad\times\bg(\varphi(R^{-1}N(t)y-z)|w_n(t,y)|^2\bg)
\,dzdxdy\nonumber\\
+&\,|B_R(0)|^{-1}\sum_{m,n}\int\int 2N(t)(\psi_{RN(t)^{-1}}(x-y)-\phi(R^{-1}N(t)|x-y|))(F(w_m)\bar w_m)(t,x)|w_n(t,y)|^2\,dxdy\nonumber\\
\geq &\,|B_R(0)|^{-1}\sum_{m,n}\int\int\int 2 N(t)\bg(\varphi(R^{-1}N(t)x-z)(F(w_m)\bar w_m)(t,x)\bg)\nonumber\\
&\qquad\qquad\qquad\qquad\qquad\qquad\qquad\qquad\qquad\times\bg(\varphi(R^{-1}N(t)y-z)|w_n(t,y)|^2\bg)
\,dzdxdy \label{galilean5}\\
+&\,o_R(1)\sum_{m}\int N(t)(F(w_m)\bar w_m)(t,x)\,dx.\label{galilean6}
\end{align}
We also notice that \eqref{galilean5} and \eqref{galilean6} are Galilean invariant. Now we consider the sum of \eqref{galilean3} and \eqref{galilean5}. Let $\chi\in C_c^\infty(\R^2;[0,1])$ be radial, decreasing, $\mathrm{supp}\, \chi\subset B_{R-\sqrt{R}}(0)$ and $\chi\equiv 1$ on $B_{R-2\sqrt{R}}(0)$. Also denote $\tilde{\chi}(x)=\chi(R^{-1}N(t)x-z)$. Then
\begin{align}
&\,4\sum_{m}\int\varphi(R^{-1}N(t)x-z)|\nabla (e^{-ix\cdot\beta(t)}w_m(t,x))|^2\,dx\nonumber\\
&\,\qquad\qquad\qquad+2\sum_{m}\int\varphi(R^{-1}N(t)x-z)(F(e^{-ix\cdot\beta(t)}w_m)\overline{e^{-ix\cdot\beta(t)}w_m})(t,x)\,dx
\label{large portion}\\
=&\,4\sum_{m}\int\tilde{\chi}(x)|\nabla (e^{-ix\cdot\beta(t)}w_m(t,x))|^2\,dx
+2\sum_{m}\int\tilde{\chi}^4(x)(F(e^{-ix\cdot\beta(t)}w_m)\overline{e^{-ix\cdot\beta(t)}w_m})(t,x)\,dx\nonumber\\
+&\,4\sum_{m}\int(\varphi-\chi)(R^{-1}N(t)x-z)|\nabla (e^{-ix\cdot\beta(t)}w_m(t,x))|^2\,dx\nonumber\\
&\,\qquad\qquad\qquad+2\sum_{m}\int(\varphi-\chi^4)(R^{-1}N(t)x-z)(F(e^{-ix\cdot\beta(t)}w_m)\overline{e^{-ix\cdot\beta(t)}w_m})(t,x)\,dx
\nonumber\\
=&\,8\bg(\frac{1}{2}\sum_{m}\int|\nabla (\tilde{\chi}(x)e^{-ix\cdot\beta(t)}w_m(t,x))|^2\,dx
+\sum_m\int\frac{1}{4}F(\tilde{\chi}e^{-ix\cdot\beta(t)}w_m)\overline{\tilde{\chi}e^{-ix\cdot\beta(t)}w_m}(t,x)\,dx\bg)
\qquad\label{longchi1}\\
+&\,4\sum_{m}\int(\varphi-\chi)(R^{-1}N(t)x-z)|\nabla (e^{-ix\cdot\beta(t)}w_m(t,x))|^2\,dx\label{longchi4}\\
&\,\qquad\qquad\qquad+2\sum_{m}\int(\varphi-\chi^4)(R^{-1}N(t)x-z)(F(e^{-ix\cdot\beta(t)}w_m)\overline{e^{-ix\cdot\beta(t)}w_m})(t,x)\,dx
\label{longchi2}\\
+&\,4\sum_m\int \tilde{\chi}(x)\mathrm{div}\,\nabla\tilde{\chi}(x)|w_m(t,x)|^2\,dx\label{longchi3}.
\end{align}
By definition we have $\varphi-\chi\geq 0$, thus $\eqref{longchi4}\geq 0$. For \eqref{longchi3}, using the definition of $\chi$ we have
\begin{align}
\eqref{longchi3}\geq -\frac{CN(t)^2}{R^3}\sum_m\int_{R-2\sqrt{R}\leq|R^{-1}N(t)x-z|\leq R-\sqrt{R}}\tilde{\chi}(x)|w_m(t,x)|^2\,dx.
\end{align}
This in turn implies
\begin{align}
&\,\int_0^T\sum_{m,n}\int\int\int\eqref{longchi3}\times \bg(\varphi(R^{-1}N(t)y-z)|w_n(t,y)|^2\bg)\,dzdxdydt\nonumber\\
\geq&\,-CR^{-3}\sum_{m,n}\int_0^T N(t)^3|w_m(t,x)|^2|w_n(t,y)|^2
\nonumber\\
&\qquad\times\bg(|B_R(0)|^{-1}\int\chi(R^{-1}N(t)x-z)\varphi(R^{-1}N(t)y-z)\,dz\bg)\,dxdydt
= o_R(1)K.
\end{align}
By Proposition \ref{gagliardo nirenberg} we know that there exists some $\vare>0$ such that
\begin{align}
\eqref{longchi1}\geq
-\vare \sum_m\int F(\tilde{\chi}e^{-ix\cdot\beta(t)}w_m)\overline{\tilde{\chi}e^{-ix\cdot\beta(t)}w_m}(t,x)\,dx.\label{longchi5}
\end{align}
Now we insert \eqref{longchi5} into the original integral and integrate over $z$, with the following observations: First, if $|x-y|\geq R^2 N(t)^{-1}$, then the supports of $\chi^4(R^{-1}N(t)x-z)$ and $\varphi(R^{-1}N(t)y-z)$ will be disjoint. Second, if $|x-y|\leq 4^{-1}R^2 N(t)^{-1}$, then
$$\inf_{|x-y|\leq 4^{-1}R^2 N(t)^{-1}}|B_R(0)|^{-1}\int\chi^4(R^{-1}N(t)x-z)\varphi(R^{-1}N(t)y-z)dz\gtrsim 1.$$
Therefore,
\begin{align}
&\,\int_0^T\sum_{m,n}\int\int\int\eqref{longchi5}\times \bg(\varphi(R^{-1}N(t)y-z)|w_n(t,y)|^2\bg)\,dzdxdydt\nonumber\\
\geq&\,-C\vare\sum_{m,n}\int_0^T N(t)\int\int_{|x-y|\leq \frac{R^2}{4N(t)}}
(F(w_m)\bar w_m)(t,x)|w_n(t,y)|^2\,dxdydt.
\end{align}
Let us finally take \eqref{longchi2}. Notice that $\varphi-\chi^4$ is supported on $|x|\in[R-2\sqrt{R},R]$, hence we obtain
\begin{align}
|B_R(0)|^{-1}\int (\varphi-\chi^4)(R^{-1}N(t)x-z)\varphi(R^{-1}N(t)y-z)dz\lesssim R^{-1},
\end{align}
which implies
\begin{align}
&\,\int_0^T\sum_{m,n}\int\int\int\eqref{longchi2}\times \bg(\varphi(R^{-1}N(t)y-z)|w_n(t,y)|^2\bg)\,dzdxdydt\nonumber\\
=&\,o_R(1)\sum_{m}\int_0^T N(t)\int\int(F(w_m)\bar w_m)(t,x)\,dxdt.
\end{align}
Summing up at this point, we have thus so far proved
\begin{align}
&\,\int_0^T\eqref{remaining 1}+\eqref{remaining 2}+\eqref{remaining 3}+\eqref{remaining 4}+\eqref{remaining 5}\,dt\nonumber\\
\geq&\,-C\vare\sum_{m,n}\int_0^T N(t)\int\int_{|x-y|\leq \frac{R^2}{4N(t)}}
(F(w_m)\bar w_m)(t,x)|w_n(t,y)|^2\,dxdydt\\
+&\,o_R(1)\sum_{m}\int_0^T N(t)\int(F(w_m)\bar w_m)(t,x)\,dxdt+o_R(1)K+\int_0^T \eqref{remaining 5}\,dt.
\end{align}
Now we recall
$$F(w_m)=-(\sum_{i}|w_i|^2+\sum_{i\neq m}|w_i|)w_m.$$
Hence for any time interval $J$ we have $\|\sum_mF(w_m)\bar w_m\|_{L_{t,x}^1(J)}\sim \|w\|^4_{\diag \ell^2(J)}$. By rewriting $w$ to $w=\vu-P_{\geq \vare_3^{-1}K}\vu$ and using Lemma \ref{arzela ascoli}, Lemma \ref{local constancy lemma}, Lemma \ref{lem before last} and conservation of mass, we conclude that if $\|\vu\|_{\diag\ell^2(J)}=1$ for some interval $J$, then
\begin{align}
&\,-\int_J\sum_{m,n}\int\int_{|x-y|\leq \frac{R^2}{4N(t)}}(F(w_m)\bar w_m)(t,x)|w_n(t,y)|^2\,dxdydt\nonumber\\
\gtrsim&\,-\int_J\sum_{m}\int(F(\vu_m)\bar \vu_m)(t,x)\,dxdt+o_R(1)+o_K(1)\nonumber\\
\gtrsim &\,\|\vu\|^4_{\diag\ell^2(J)}+o_R(1)+o_K(1)= 1+o_R(1)+o_K(1)
\end{align}
as $R,K\to\infty$. We now partition $[0,T]$ into $[0,T]=\cup J_{\ell}$ such that $\|\vu\|_{\diag\ell^2(J_{\ell})}=1$ for all $J_{\ell}$. Then for sufficiently large $R$ and $K$
\begin{align}
&\,-\int_0^T N(t)\sum_{m,n}\int\int_{|x-y|\leq \frac{R^2}{4N(t)}}(F(w_m)\bar w_m)(t,x)|w_n(t,y)|^2\,dxdydt\nonumber\\
\sim&\sum_{J_{\ell}}-N(J_L)\int_{J_L}\sum_{m,n}\int\int_{|x-y|\leq \frac{R^2}{4N(t)}}(F(w_m)\bar w_m)(t,x)|w_n(t,y)|^2\,dxdydt\nonumber\\
\gtrsim&\,\sum_{J_{\ell}}N(J_{\ell})\sim\int_0^TN(t)^3\,dt=K.
\end{align}
In the same manner,
\begin{align}
o_R(1)\sum_{m}\int_0^T N(t)\int(F(w_m)\bar w_m)(t,x)\,dxdt=o_R(1)K.
\end{align}
Let us finally take \eqref{remaining 5}. Direct calculation shows
\begin{align}
&\,\eqref{remaining 5}\nonumber\\
=&\,2\sum_{m,n}\int\int N'(t)\phi(R^{-1}N(t)|x-y|)(x_j-y_j)\mathrm{Im}(\bar w_m\pt_j w_m)(t,x)|w_n(t,y)|^2\,dxdy\nonumber\\
=&\,2|B_R(0)|^{-1}\sum_{m,n}\int\int\int\varphi(R^{-1}N(t)x-z)\varphi(R^{-1}N(t)y-z)\nonumber\\
&\qquad\qquad\qquad\times  N'(t)(x_j-y_j)\mathrm{Im}(\bar w_m\pt_j w_m)(t,x)|w_n(t,y)|^2\,dzdxdy.\label{longchi6}
\end{align}
Again, in order to keep the supports of $\varphi(R^{-1}N(t)x-z)$ and $\varphi(R^{-1}N(t)y-z)$ to be not disjoint, it is necessary that $|x-y|\leq R^2 N(t)^{-1}$. Moreover, one easily verifies that \eqref{longchi6} is Galilean invariant. Combining with Young's inequality we infer that for arbitrary $\gamma>0$ there exists some $C(\gamma)>0$ such that
\begin{align}
&\,\eqref{remaining 5}\nonumber\\
\lesssim &\,|B_R(0)|^{-1}\gamma N(t)
\sum_{m,n}\int\int\int\varphi(R^{-1}N(t)x-z)\varphi(R^{-1}N(t)y-z)\nonumber\\
&\qquad\qquad\qquad\times  |\nabla(e^{-ix\cdot\beta(t)}w_m)(t,x)|^2|w_n(t,y)|^2\,dzdxdy\label{absorb}\\
+ &\,R^4|B_R(0)|^{-1} C(\gamma)N(t)^{-3}N'(t)^2
\sum_{m,n}\int\int\int\varphi(R^{-1}N(t)x-z)\varphi(R^{-1}N(t)y-z)\nonumber\\
&\qquad\qquad\qquad\times  |w_m(t,x)|^2|w_n(t,y)|^2\,dzdxdy.\label{decreasing}
\end{align}
\eqref{absorb} can be absorbed to the first term in \eqref{large portion} by choosing $\gamma$ small. For \eqref{decreasing}, using conservation of mass we have
\begin{align}
\int_0^T\eqref{decreasing}\,dt\lesssim C(\gamma)R^4\int_0^T|N'(t)|\,dt.\label{estimate}
\end{align}
In general the best we can hope for estimating \eqref{estimate} would be $\int_0^T |N'(t)|\,dt\lesssim \int_0^T N(t)^3\,dt= K$, which is obviously insufficient for our purpose. The idea is to modify $N(t)$ in a way such that $N(t)$ is less oscillatory so that $|N'(t)|$ is small. This can be done using the so-called smoothing algorithm initiated by Dodson \cite[Sec. 6.1]{dodson2d}, where the level of the peaks of the function $N(t)$ are inductively reduced. The adaptation of the smoothing method to our model is however verbatim, we thus omit the details here. In a nutshell, after applying the smoothing algorithm we may replace the function $N(t)$ by a new function $\tilde{N}(t)$ such that $\tilde{N}(t)\leq N(t)$ and
$$\int_0^T|\tilde{N}'(t)|\,dt\ll K$$
for sufficiently large $K$. Now we fix $R=R(\vare)$, then applying the smoothing algorithm to get a new function $\tilde{N}(t)$ that is determined by the number $R$. Since $\tilde{N}(t)\leq N(t)$, the number $R$ is not affected. Possibly we also need to shrink the value of $\vare_3$, but this does not effect the results from Lemma \ref{lem small K} since nevertheless we need to take $K$ very large. Summing up, we conclude
\begin{align}
&\,\int_0^T\eqref{remaining 1}+\eqref{remaining 2}+\eqref{remaining 3}+\eqref{remaining 4}+\eqref{remaining 5}\,dt\nonumber\\
\geq&\,2^{-1}C\vare K
+o_R(1)K+o_K(1)K+R^4o_K(1)K.
\end{align}
as $R,K\to\infty$. The desired claim follows by firstly taking $R=R(\vare)$ sufficiently large, then modifying the frequency scale function $N(t)$, shrinking the number $\vare_3$, and finally taking $K=K(R)$ sufficiently large to apply Lemma \ref{lem small K}.
\end{proof}

\begin{proof}[Proof of Theorem \ref{thm large scale resonant}]
This follows immediately from Lemma \ref{impos rapid cascade} and Lemma \ref{lem quasi soliton}.
\end{proof}

\section{Scattering for the focusing cubic NLS on $\R^2\times \T$}\label{section: scattering r2t1}
In this final section we give the proof of Theorem \ref{main thm}. Up to the variational and virial analysis, many of the arguments given in this section are similar to the ones from \cite{CubicR2T1Scattering}, where the authors studied the defocusing analogue of \eqref{cnls}. Nonetheless, the linear profile decomposition established in \cite{CubicR2T1Scattering} is insufficient for our purpose: the linear profile decomposition in \cite{CubicR2T1Scattering} is given at the $L_x^2 H_y^1$-level, while in our case we need a linear profile decomposition at the $H_{x,y}^1$-level in order to apply the variational arguments. We follow the same lines in \cite{KillipVisanKleinGordon,killip_visan_soliton,luo_double_crit} to construct such a linear profile decomposition.

\subsection{Small data well-posedness and stability theories}
We collect in this section the small data and stability theories for \eqref{cnls} and some useful inequalities.

\begin{lemma}[Strichartz estimates on $\R^2\times\T$, \cite{TzvetkovVisciglia2016}]\label{strichartz}
Let $\gamma\in\R$, $s\in[0,1)$ and $p,q,\tilde{p},\tilde{q}$ satisfy $p,\tilde{p}\in(2,\infty]$ and $2p^{-1}+2q^{-1}
=2\tilde{p}^{-1}+2\tilde{q}^{-1}=1-s$. Then for a time interval $I\ni t_0$ we have
\begin{align}
\|e^{i(t-t_0)\dxy}f\|_{L_t^p L_x^q H^\gamma_y(I)}&\lesssim\|f\|_{{H}_x^sH_y^\gamma}.
\end{align}
Moreover, the Strichartz estimate for the Duhamel term
\begin{align}
\|\int_{t_0}^t e^{i(t-s)\dxy}F(s)\,ds\|_{L_t^p L_x^q H^\gamma_y(I)}&\lesssim \|F\|_{L_t^{\tilde{p}'} L_x^{\tilde{q}'} H^\gamma_y(I)}
\end{align}
holds in the case $s=0$.
\end{lemma}

\begin{lemma}[Fractional calculus on $\T$, \cite{TzvetkovVisciglia2016}]\label{fractional lemma} For $s\in(\frac12,1]$ we have
\begin{align}
\|u_1 u_2 \|_{\dot{H}^s_y}&\lesssim \|u_1\|_{\dot{H}^s_y}\|u_2\|_{L_y^\infty}+\|u_2\|_{\dot{H}^s_y}\|u_1\|_{L_y^\infty},\\
\|u_1 u_2 u_3\|_{\dot{H}^s_y}&\lesssim\sum_{i=1}^3\|u_i\|_{\dot{H}^s_y}\prod_{j=1,j\neq i}^3\|u_j\|_{L^\infty_y}.
\end{align}
\end{lemma}

\begin{lemma}[Small data well-posedness, \cite{TzvetkovVisciglia2016}]\label{lemma cnls well posedness} Let $I$ be an open interval containing $0$. Define
\begin{align}
X(I)&:=(L_t^\infty L_x^2 {H}_y^1(I)\cap L_t^{2^+} L_x^{\infty^-} {H}_y^1(I))\cap (L_t^\infty \dot{H}_x^1 L_y^2(I)
\cap L_t^{2^+} \dot{W}_x^{1,\infty^-} L_y^2(I))\nonumber\\
&=:S_0 {H}_y^1(I)\cap S_1 L_y^2(I).
\end{align}
Let also $s\in(\frac12,1]$. Assume that
\begin{align}
\|U_0\|_{H_{x,y}^1}\leq A
\end{align}
for some $A>0$. Then there exists $\delta=\delta(A)$ such that if
\begin{align}
\|e^{it\Delta}U_0\|_{\diag  H_y^{s}(I) }\leq \delta,
\end{align}
then there exists a unique solution $U\in X(I)$ of \eqref{cnls} with $U(0)=U_0$ such that
\begin{align}
\|U\|_{X(I)}&\lesssim A,\\
\|U\|_{\diag  H_y^{s}(I)}&\leq 2\|e^{it\Delta}U_0\|_{\diag  H_y^{s}(I)}.
\end{align}
\end{lemma}

\begin{lemma}[Scattering criterion]\label{scattering crit}
If $U$ is a global solution of \eqref{cnls} and there exists some $s\in(\frac12,1]$ such that
\begin{align}
\|U\|_{L_{t,x}^4 H_y^{s}(\R)}+\|U\|_{L_t^\infty H_{x,y}^1(\R)}<\infty,\label{scattering threshold}
\end{align}
then $U$ is scattering in $H_{x,y}^1$. Moreover, we have
\begin{align}
\|U\|_{L_{t,x}^4 H_y^{1}(\R)}\leq C(\|U\|_{L_{t,x}^4 H_y^{s}(\R)},\|U\|_{L_t^\infty H_{x,y}^1(\R)}).\label{scattering threshold 2}
\end{align}
\end{lemma}

\begin{proof}
That \eqref{scattering threshold} implies scattering was proved by \cite[Thm. 2.9]{CubicR2T1Scattering}. Now using Duhamel's formula and Strichartz estimate we also infer that \eqref{scattering threshold} implies
$$\|U\|_{L_t^\frac{8}{3} L_x^{8} H_y^s(\R)}\leq C(\|U\|_{L_{t,x}^4 H_y^{s}(\R)},\|U\|_{L_t^\infty H_{x,y}^1(\R)})<\infty.$$
Then by Duhamel's formula, Strichartz and H\"older
\begin{align}
\|U\|_{L_{t,x}^4 H_y^{1}(\R)}
&\lesssim \|U\|_{L_t^\infty H_{x,y}^1(\R)}+\||U|^3\|_{L_{t,x}^{\frac43}L_y^2(\R)}+\||U|^2\pt_y U\|_{L_{t,x}^{\frac43}L_y^2(\R)}\nonumber\\
&\lesssim\|U\|_{L_t^\infty H_{x,y}^1(\R)}(1+\|U\|^2_{L_t^\frac{8}{3} L_x^{8} H_y^s(\R)}),
\end{align}
which implies \eqref{scattering threshold 2}.
\end{proof}

\begin{lemma}[Long time stability, \cite{CubicR2T1Scattering}]\label{lem stability cnls}
Let $U$ be a solution of \eqref{cnls} on the time interval $I\ni 0$ and let $Z$ be a solution of
\begin{align}
i\pt_t Z+\Delta_{x,y} Z=-|Z|^2 Z+e
\end{align}
on $I$. Let also $s\in(\frac12,1]$ be given. Assume that
\begin{align}
\|U\|_{L_t^\infty L_x^2 H_y^{s}(I)}&\leq M,\label{cond1 cnls}\\
\|Z\|_{\diag H_y^{s} (I)}&\leq L,\label{cond2 cnls}\\
\|Z(0)-U(0)\|_{L_x^2 H_y^{s}}&\leq M'.\label{cond3 cnls}
\end{align}
Assume also the smallness conditions
\begin{alignat}{2}
\|e^{it\Delta}(Z(0)-U(0) )&\|_{\diag H_y^{s}(I)}&&\leq \vare,\label{small initial long cnls}\\
\|e&\|_{L_{t,x}^{\frac{4}{3}}H_y^{s}(I)}&&\leq \vare\label{small error long cnls}
\end{alignat}
for some $0<\vare\leq \vare_1$ where $\vare_0=\vare_0(M,M',L)>0$ is a small constant. Then
\begin{align}
\|Z-U\|_{\diag H_y^{s}(I)}&\leq C(M,M',L)\vare,\label{long erro1 cnls}\\
\|Z-U\|_{S_0H_y^{s}(I)}&\leq C(M,M',L)M',\label{long erro2 cnls}\\
\|U\|_{S_0H_y^{s}(I)}&\leq C(M,M',L)\label{long erro3 cnls}.
\end{align}
\end{lemma}

\subsection{Linear profile decomposition}
In this section we establish a linear profile decomposition for a bounded sequence in $H_{x,y}^1$. Firstly we fix some notation. For each $j\in\Z$, define $\mC_j$ by
\begin{align*}
\mC_j:=\bg\{\Pi_{i=1}^2 [2^jk_i,2^j(k_i+1))\subset\R^2:k\in\Z^2\bg\}
\end{align*}
and $\mC:=\cup _{j\in\Z}\,\mC_j$. Given $Q\in\mC$ we define $f_Q$ by $\mathcal{F}_x (f_Q):=\chi_Q \mathcal{F}_x f$, where $\chi_Q$ is the characteristic function of the cube $Q$.

\begin{lemma}[Improved Strichartz estimate, \cite{CubicR2T1Scattering}]\label{cnls lem improved strichartz}
For $f\in H_{x,y}^1(\R^2\times\T)$, we have the following refined Strichartz estimate
\begin{align}
\|e^{it\Delta_x}f\|_{L_{t,x,y}^4(\R)}\lesssim \|f\|^{\frac34}_{L_x^2 H_y^1}
\bg(\sup_{Q\in\mC}|Q|^{-\frac{3}{22}}\|e^{it\Delta_x}f_Q\|_{L_{t,x,y}^{\frac{11}{2}}(\R)}\bg)^{\frac{1}{4}}.
\end{align}
\end{lemma}

\begin{lemma}[Inverse Strichartz inequality]\label{refined l2 lemma 1}
Let $(f_n)_n\subset H_{x,y}^1(\R^2\times\T)$. Suppose that
\begin{align}
\lim_{n\to\infty}\|f_n\|_{H_{x,y}^1}=A<\infty\quad\text{and}\quad\lim_{n\to\infty}\|e^{it\Delta_x} f_n\|_{L^4_{t,x,y}(\R)}=\vare>0.
\end{align}
Then up to a subsequence, there exist $\phi\in L_x^2H_y^1(\R^2\times\T)$ and $(t_n,x_n,\xi_n,\ld_n)_n\subset\R\times\R^2\times\R^2\times(0,\infty)$ such that $\limsup_{m\to\infty}|\xi_n|<\infty$ and $\lim_{n\to\infty}\ld_n=:\ld_\infty\in(0,\infty]$. Moreover,
\begin{align}\label{cnls l2 refined strichartz basic 5}
&\,\,\ld_n e^{-i\xi_n\cdot(\ld_n x+x_n)}(e^{it_n\Delta}f_n)(\ld_n x+x_n,y)\nonumber\\
\nonumber\\
\rightharpoonup &\,\,\phi(x,y)\text{ weakly in }
\left\{
             \begin{array}{ll}
             H_{x,y}^1(\R^2\times\T),&\text{if $\limsup_{n\to\infty}|\ld_n\xi_n|<\infty$},\\
             \\
             L_x^2H_y^1(\R^2\times\T),&\text{if $|\ld_n\xi_n|\to\infty$}.
             \end{array}
\right.
\end{align}
Additionally, if $\limsup_{n\to\infty}|\ld_n\xi_n|<\infty$, then $\xi_n\equiv 0$. Setting
\begin{align}
\phi_n:=
\left\{
             \begin{array}{ll}
             \ld_n^{-1}e^{-it_n\Delta_x}\bg[\phi(\frac{x-x_n}{\ld_n},y)\bg],&\text{if $\ld_\infty<\infty$},\\
             \\
             \ld_n^{-1}e^{-it_n\Delta_x}\bg[e^{i\xi_n\cdot x}
             (P_{\leq\ld_n^\theta}\phi)(\frac{x-x_n}{\ld_n},y)\bg],&\text{if $\ld_\infty=\infty$}
             \end{array}
\right.
\end{align}
for some fixed $\theta\in(0,1)$, we have
\begin{align}
\lim_{n\to\infty}&(\|f_n\|^2_{L_x^2 H_y^1}-\|f_n-\phi_n\|^2_{L_x^2 H_y^1})=\|\phi\|^2_{L_x^2 H_y^1}\gtrsim A^2\bg(\frac{\vare}{A}\bg)^{24},
\label{cnls l2 refined strichartz decomp 1}\\
\lim_{n\to\infty}&(\|f_n\|^2_{{H}_x^1L_y^2}-\|f_n-\phi_n\|^2_{{H}_x^1L_y^2}-\|\phi_n\|^2_{{H}_x^1L_y^2})=0,
\label{cnls l2 refined strichartz decomp 3}\\
\lim_{n\to\infty}&(\|f_n\|^2_{L_{x,y}^2}-\|f_n-\phi_n\|^2_{L_{x,y}^2}-\|\phi_n\|^2_{L_{x,y}^2})=0,\label{cnls l2 refined strichartz decomp 2}\\
\lim_{n\to\infty}&(\|f_n\|^2_{L_x^2 {H}_y^1}-\|f_n-\phi_n\|^2_{L_x^2{H}_y^1}-\|\phi_n\|^2_{L_x^2{H}_y^1})=0.
\label{cnls l2 refined strichartz decomp 4}
\end{align}
\end{lemma}

\begin{proof}
For $R>0$, denote by $f^R$ the function such that $\mathcal{F}_x(f^R)=\chi_R\mathcal{F}_x f$, where $\chi_R$ is the characteristic function of the ball $B_R(0)\subset \R^2$. First we obtain that
\begin{align}\label{spare0}
\sup_{n\in\N}\|f_n-f_n^R\|_{L_x^2 H_y^1}^2
&=\sup_{n\in\N}\sum_k\la k\ra^2\int_{|\xi|\geq R}|\mathcal{F}_{x,y}f_n(\xi,k)|^2\,d\xi\nonumber\\
&\leq R^{-2}\sup_{n\in\N}\sum_k\la k\ra^2\int|\xi|^2|\mathcal{F}_{x,y}f_n(\xi,k)|^2\,d\xi\nonumber\\
&=R^{-2}\sup_{n\in\N}\|f_n\|^2_{H_{x,y}^1}\lesssim R^{-2}A^2\to 0
\end{align}
as $R\to\infty$. Combining with Strichartz and the embedding $H_y^1\hookrightarrow L_y^4$, we infer that there exists some $K_1>0$ such that for all $R\geq K_1$ one has
\begin{align*}
\sup_{n\in\N}\|f^R_n\|_{L_x^2 H_y^1}\lesssim A\quad\text{and}\quad\sup_{n\in\N}\|e^{it\Delta_x} f^R_n\|_{L_{t,x,y}^4(\R)}\gtrsim \vare.
\end{align*}
Applying Lemma \ref{cnls lem improved strichartz} to $(f^R_n)_n$, we know that there exists $(Q_n)_n\subset\mC$ such that
\begin{align}\label{l2 refined strichartz basic 1}
\vare^{4}A^{-3}\lesssim \inf_{n\in\N}|Q_n|^{-\frac{3}{22}}\|e^{it\Delta_x}(f^R_n)_{Q_n}\|_{L^{\frac{11}{2}}_{t,x,y}(\R)}.
\end{align}
Let $\ld_n^{-1}$ be the side-length of $Q_n$. Denote also by $\xi_n$ the center of $Q_n$. Using Strichartz we obtain that
\begin{align*}
\sup_{n\in\N}\|e^{it\Delta_x}(f^R_n)_{Q_n}\|_{L^{\frac{11}{2}}_{t,x,y}(\R)}
&\lesssim\sup_{n\in\N}\|e^{it\Delta_{x,y}}(e^{-it\Delta_y}(f^R_n)_{Q_n})\|_{L^{\frac{11}{2}}_{t,x} H_y^{\frac{8}{11}}(\R)}\nonumber\\
&\lesssim \sup_{n\in\N}\|f_n\|_{{H}_x^{\frac{3}{11}}H_y^{\frac{8}{11}}}\lesssim
\sup_{n\in\N}(\|f_n\|^{\frac{3}{11}}_{H_x^{1}L_y^2}\|f_n\|^{\frac{8}{11}}_{L_x^2 H_y^1})
\lesssim \sup_{n\in\N}\|f_n\|_{H_{x,y}^1}\lesssim A,
\end{align*}
which in turn implies $\sup_{n\in\N}|Q_n|\lesssim 1$. Since $(\mathcal{F}_x(f_n^R))_n$ are supported in $B_R(0)$, we may assume that $(Q_n)_n\subset B_{R'}(0)$ for some sufficiently large $R'>0$. Therefore $(\ld_n)_n$ is bounded below and $(\xi_n)_n$ is bounded in $\R^2$. H\"older, Strichartz and the embedding $H_y^1\hookrightarrow L_y^4 $ yield
\begin{align*}
|Q_n|^{-\frac{3}{22}}\|e^{it\Delta_x}(f^R_n)_{Q_n}\|_{L^{\frac{11}{2}}_{t,x,y}(\R)}
&\lesssim\ld_n^{\frac{3}{11}}\|e^{it\Delta_x}(f^R_n)_{Q_n}\|^{\frac{8}{11}}_{L_{t,x,y}^4(\R)}
\|e^{it\Delta_x}(f^R_n)_{Q_n}\|^{\frac{3}{11}}_{L_{t,x,y}^\infty(\R)}
\nonumber\\
&\lesssim \ld_n^{\frac{3}{11}}\vare^{\frac{8}{11}}\|e^{it\Delta_x}(f^R_n)_{Q_n}\|^{\frac{3}{11}}_{L_{t,x,y}^\infty(\R)}.
\end{align*}
Combining with \eqref{l2 refined strichartz basic 1} we infer that there exist $(t_n,x_n,y_n)_n\subset\R\times\R^d\times \T$ such that
\begin{align}\label{l2 refined strichartz basic 2}
\liminf_{n\to\infty}\ld_n|[e^{it_n\Delta_x}(f^R_n)_{Q_n}](x_n,y_n)|\gtrsim \vare^{12}A^{-11}.
\end{align}
Since $\T$ is compact, we simply assume that $y_n\equiv 0$. Define
\begin{align*}
h_n(x,y)&:=\ld_ne^{-i\xi_n(\ld_n x+x_n)}(e^{it_n\Delta}f_n)(\ld_n x+x_n,y),\\
h^R_n(x,y)&:=\ld_ne^{-i\xi_n(\ld_n x+x_n)}(e^{it_n\Delta}f^R_n)(\ld_n x+x_n,y).
\end{align*}
It is easy to verify that $\|h_n\|_{L_x^2 H_y^1}=\|f_n\|_{L_x^2 H_y^1}$. By the $L_x^2 H_y^1$-boundedness of $(f_n)_n$ we know that there exists some $\phi\in L_x^2 H_y^1$ such that $h_n\rightharpoonup \phi$ weakly in $L_x^2 H_y^1$. Arguing similarly, we infer that $(h^R_n)_n$ converges weakly to some $\phi^R\in L_x^2 H_y^1$. By definition of $\phi$ and $\phi^R$ we see that
\begin{align*}
\|\phi-\phi^R\|^2_{L_x^2 H_y^1}=\lim_{n\to\infty}\la h_n-h_n^R,\phi-\phi^R\ra_{L_x^2 H_y^1}\leq (\limsup_{n\to\infty}\|h_n-h_n^R\|_{L_x^2 H_y^1})\|\phi-\phi^R\|_{L_x^2 H_y^1}.
\end{align*}
Using \eqref{spare0} we then obtain that
\begin{align}\label{spare2}
\phi^R\to \phi\quad\text{in $L_x^2 H_y^1$ as $R\to\infty$}.
\end{align}
Now define the function $\chi_1$ such that $\mathcal{F}_x\chi_1$ is the characteristic function of the cube $[-\frac{1}{2},\frac{1}{2})^2\subset\R^2$. Also let $\chi_2=\delta_0\in H^{-1}_y$, where $\delta_0$ is the Dirac function at zero. Since $H_y^1\hookrightarrow L_y^\infty$, we infer that $\|\chi_2\|_{H_y^{-1}}\lesssim 1$. From \eqref{l2 refined strichartz basic 2}, the weak convergence of $h^R_n$ to $\phi^R$ in $L_x^2 H_y^1$ and change of variables it follows
\begin{align}\label{l2 refined strichartz basic 3}
\la\phi^R,\chi_1\chi_2\ra_{L_{x,y}^2}=\lim_{n\to\infty}\ld_n^{\frac{d}{2}}|[e^{it_n\Delta}(f^R_n)_{Q_n}](x_n,0)|
\gtrsim \vare^{12}A^{-11}.
\end{align}
On the other hand, using duality we infer that
\begin{align*}
|\la\phi^R,\chi_1\chi_2\ra_{L_{x,y}^2}|\leq\|\phi^R\|_{L_x^2 H_y^1}\|\chi_1\|_{L_x^2}\|\chi_2\|_{H_y^{-1}}\lesssim \|\phi^R\|_{L_x^2 H_y^1}.
\end{align*}
Thus
\begin{align}\label{spare3}
\|\phi^R\|_{L_x^2 H_y^1}^2\geq C\vare^{12}A^{-11}
\end{align}
for some $C=C(d)>0$ which is uniform for all $R\geq K_1$. Now using \eqref{spare2} and \eqref{spare3} we finally deduce that
\begin{align}
\|\phi\|_{L_x^2 H_y^1}^2 \geq\|\phi^R\|_{L_x^2 H_y^1}^2-\frac{C}{2}\vare^{12}A^{-11}
\geq \frac{C}{2}\vare^{12}A^{-11}
\end{align}
for sufficiently large $R$, which gives the lower bound of \eqref{cnls l2 refined strichartz decomp 1}. From now on we fix $R$ such that the lower bound of \eqref{cnls l2 refined strichartz decomp 1} is valid for this chosen $R$ and let $(t_n,x_n,\xi_n,\ld_n)_n$ be the corresponding symmetry parameters. Since $L_x^2 H_y^1$ is a Hilbert space, from the weak convergence of $h_n$ to $\phi$ in $L_x^2 H_y^1$ we obtain that
\begin{align*}
\lim_{n\to\infty}(\|h_n\|^2_{L_x^2 H_y^1}-\|\phi\|^2_{L_x^2 H_y^1}-\|h_n-\phi\|^2_{L_x^2 H_y^1})
=2\lim_{n\to\infty}\mathrm{Re}\,\la \phi,h_n-\phi\ra_{L_x^2 H_y^1}=0.
\end{align*}
Combining with the fact that
\begin{align*}
\|P_{\leq\ld_n^\theta}\phi-\phi\|_{L_x^2 H_y^1}\to 0\quad\text{as $n\to\infty$}
\end{align*}
for $\ld_n\to\infty$ we deduce the equalities in \eqref{cnls l2 refined strichartz decomp 1} and \eqref{cnls l2 refined strichartz decomp 4}. Since $L_{x,y}^2\subset L_x^2 H_y^1$ is also a Hilbert space, \eqref{cnls l2 refined strichartz decomp 2} follows verbatim. In the case $\limsup_{n\to\infty}|\ld_n\xi_n|<\infty$, using the boundedness of $(\ld_n\xi_n)_n$ and chain rule, we also infer that $\|h_n\|_{H_{x,y}^1}\lesssim \|f_n\|_{H_{x,y}^1}$. By the $H_{x,y}^1$-boundedness of $(f_n)_n$ and uniqueness of weak limit we deduce additionally that $\phi\in H_{x,y}^1$ and \eqref{cnls l2 refined strichartz basic 5} follows.

Next we show that we may assume $\xi_n\equiv 0$ under the additional condition $\limsup_{n\to\infty}|\ld_n\xi_n|<\infty$. Define
\begin{align*}
\mathcal{T}_{a,b}u(x):=be^{ia\cdot x}u(x)
\end{align*}
for $a\in\R^d$ and $b\in\C$ with $|b|=1$. Let also
\begin{align*}
(\ld\xi)_\infty&:=\lim_{n\to\infty}\ld_n\xi_n,\\
e^{i(\xi\cdot x)_\infty}&:=\lim_{n\to\infty}e^{i\xi_n\cdot x_n}.
\end{align*}
By the boundedness of $(\ld_n\xi_n)_n$ we infer that $\mathcal{T}_{\ld_n\xi_n,e^{i\xi_n\cdot x_n}}$ is an isometry on $L_{x,y}^2$ and converges strongly to $\mathcal{T}_{(\ld\xi)_\infty,e^{i(\xi\cdot x)_\infty}}$ as operators on $H_{x,y}^1$. We may replace $h_n$ by $\ld_n(e^{it_n\Delta}f_n)(\ld_n x+x_n,y)$ and $\phi$ by $\mathcal{T}_{(\ld\xi)_\infty,e^{i(\xi\cdot x)_\infty}}\phi$ and \eqref{cnls l2 refined strichartz basic 5}, \eqref{cnls l2 refined strichartz decomp 1} and \eqref{cnls l2 refined strichartz decomp 3} carry over.

Finally, we prove \eqref{cnls l2 refined strichartz decomp 3}. In the case $\ld_\infty<\infty$ we additionally know that $\phi\in H_{x,y}^1$ and $\xi_n\equiv 0$. Using the fact that $H_x^1 L_y^2 $ is a Hilbert space, \eqref{cnls l2 refined strichartz decomp 2} and change of variables we obtain
\begin{align*}
o_n(1)=\|h_n\|_{\dot{H}_x^1 L_y^2}-\|h_n-\phi\|_{\dot{H}_x^1 L_y^2}-\|\phi\|_{\dot{H}_x^1 L_y^2}=
\ld_n^2(\|f_n\|_{\dot{H}_x^1 L_y^2}-\|f_n-\phi_n\|_{\dot{H}_x^1 L_y^2}-\|\phi_n\|_{\dot{H}_x^1 L_y^2}).
\end{align*}
Combining with the lower boundedness of $(\ld_n)_n$, this implies that
\begin{align*}
\|f_n\|_{\dot{H}_x^1 L_y^2}-\|f_n-\phi_n\|_{\dot{H}_x^1 L_y^2}-\|\phi_n\|_{\dot{H}_x^1 L_y^2}=\ld_n^{-2}o_n(1)=o_n(1),
\end{align*}
which gives \eqref{cnls l2 refined strichartz decomp 3} in the case $\ld_\infty<\infty$. Assume now $\ld_\infty=\infty$. Using change of variables and chain rule we obtain that
\begin{align}
&\,\|f_n\|^2_{\dot{H}_x^1L_y^2}-\|f_n-\phi_n\|^2_{\dot{H}_x^1L_y^2}-\|\phi_n\|^2_{\dot{H}_x^1L_y^2}\nonumber\\
=&\,|\xi_n|^2\bg(\|h_n\|_{L^2_{x,y}}^2-\|h_n-P_{\leq \ld_n^\theta}\phi\|^2_{L^2_{x,y}}-\|P_{\leq \ld_n^\theta}\phi\|^2_{L^2_{x,y}}\bg)\nonumber\\
&+2\ld_n^{-1}\mathrm{Re}\bg(\la i\xi_n (h_n-P_{\leq \ld_n^\theta}\phi),\nabla_x P_{\leq \ld_n^\theta}\phi\ra_{L^2_{x,y}}
+\la i\xi_n P_{\leq \ld_n^\theta}\phi,\nabla_x (h_n-P_{\leq \ld_n^\theta}\phi)\ra_{L^2_{x,y}}\bg)\nonumber\\
&+\ld_n^{-2}\bg(\|h_n\|^2_{\dot{H}_x^1L_y^2}-\|h_n-P_{\leq \ld_n^\theta}\phi\|^2_{\dot{H}_x^1L_y^2}-\|P_{\leq \ld_n^\theta}\phi\|^2_{\dot{H}_x^1L_y^2}\bg)\nonumber\\
=&: I_1+I_2+I_3.
\end{align}
Using the boundedness of $(\xi_n)_n$ and \eqref{cnls l2 refined strichartz decomp 2} we infer that $I_1\to 0$. For $I_2$, using Bernstein and the boundedness of $(\xi_n)_n$ in $\R^d$ and of $(h_n-P_{\leq \ld_n^\theta}\phi)_n$ in $L_{x,y}^2$ we see that
\begin{align*}
|I_2|\lesssim \ld_n^{-1}\|h_n-P_{\leq \ld_n^\theta}\phi\|_{L_{x,y}^2}
\|\nabla_x P_{\leq \ld_n^\theta}\phi\|_{L_{x,y}^2}\lesssim \ld_n^{-(1-\theta)}\to 0.
\end{align*}
Finally, $I_3$ can be similarly estimated using Bernstein inequality, we omit the details here. Summing up we conclude \eqref{cnls l2 refined strichartz decomp 2}.
\end{proof}

\begin{remark}
By redefining the symmetry parameters we may w.l.o.g. assume that
\begin{align}
\text{(i) } &\ld_n\equiv 1\quad\text{or}\quad \ld_n\to \infty,\\
\text{(ii) } &t_n\equiv 0\quad\text{or}\quad \frac{t_n}{\ld_n^2}\to \pm\infty
\end{align}
and the linear profiles $\phi_n$ take the form
\begin{align*}
\phi_n=
\left\{
             \begin{array}{ll}
             e^{it_n\Delta_x}\phi(x-x_n,y),&\text{if $\ld_\infty=1$},\\
             \\
             \ld_n^{-1}e^{i x\cdot\xi_n}[e^{it_n\Delta_x}P_{\leq\ld_n^\theta}\phi](\ld_n^{-1}(x-x_n),y),&\text{if $\ld_\infty=\infty$}.
             \end{array}
\right.
\end{align*}
\end{remark}

\begin{lemma}
We have
\begin{align}
&\|f_n\|_{L_{x,y}^4}^4=\|\phi_n\|_{L_{x,y}^4}^4+\|f_n-\phi_n\|_{L_{x,y}^4}^4+o_n(1).\label{decomp tas}
\end{align}
\end{lemma}

\begin{proof}
Assume first that $\ld_\infty=1$ and $t_n\to\pm\infty$. In this case, we have $\phi\in H_{x,y}^1$. For $\beta>0$ let $\psi\in C_c^\infty (\R^2)\otimes C_{\mathrm{per}}^\infty(\T)$ such that $\|\phi-\psi\|_{H_{x,y}^1}\leq\beta$. Define also $\psi_n:=e^{it_n\Delta_x}\psi(x-x_n,y)$. Then by dispersive estimate we deduce that
\begin{align*}
\|\psi_n\|_{L_{x,y}^4}\lesssim |t_n|^{-\frac12}\|\psi\|_{L_{y}^4 L_x^{\frac43}}\to 0.
\end{align*}
Now let $\zeta\in C^\infty(\R^3;[0,1])$ be a cut-off function such that $\mathrm{supp}\,\zeta\subset \R^2\times [-2\pi,2\pi]$ and $\zeta\equiv 1$ on $\R^2\times [-\pi,\pi]$. Then by Gagliardo-Nirenberg inequality, product rule and periodicity along the $y$-direction we have
\begin{align}
\|\psi_n-\phi_n\|_{L_{x,y}^4(\R^2\times\T)}
\leq\|\zeta(\psi_n-\phi_n)\|_{L_{x,y}^4(\R^3)}
\lesssim \|\zeta(\psi_n-\phi_n)\|_{H_{x,y}^1(\R^3)}\lesssim \|\psi_n-\phi_n\|_{H_{x,y}^1(\R^2\times\T)}\leq\beta,\label{argument gn}
\end{align}
which in turn implies $\|\phi_n\|_{L_{x,y}^4}=o_n(1)$. Therefore by triangular inequality
\begin{align*}
|\|f_n\|_{L_{x,y}^4}-\|f_n-\phi_n\|_{L_{x,y}^4}|\leq \|\phi_n\|_{L_{x,y}^4}=o_n(1)
\end{align*}
and \eqref{decomp tas} follows. Now we assume $\ld_\infty=1$ and $t_n\equiv 0$. Then we use the Brezis-Lieb lemma to deduce
\begin{align*}
\|h_n\|_{L_{x,y}^4}^4 =\|\phi\|_{L_{x,y}^4}^4+\|h_n-\phi\|_{L_{x,y}^4}^4+o_n(1).
\end{align*}
\eqref{decomp tas} follows then by undoing the transformation. Finally, we take the case $\ld_\infty=\infty$. Using Gagliardo-Nirenberg, chain rule, Bernstein, Minkowski and the embedding $H_y^1\hookrightarrow L_y^4$
\begin{align}
\|\phi_n\|_{L_{x,y}^4}&\leq \|\|\phi_n\|^{\frac12}_{L_x^2}(\ld_n^{-\frac12})\|\nabla_x(P_{\leq \ld_n^\theta}\phi)\|^{\frac12}_{L_x^2}\|_{L_y^4}\nonumber\\
&\lesssim \ld_n^{-\frac{1-\theta}{2}}\|\phi_n\|_{L_y^4L_x^2}\lesssim  \ld_n^{-\frac{1-\theta}{2}}\|\phi_n\|_{L_x^2 H_y^1}\to\0
\end{align}
as $n\to\infty$. The desired claim then follows again by triangular inequality.
\end{proof}

\begin{lemma}[Linear profile decomposition for bounded $H_{x,y}^1$-sequence]\label{linear profile}
Let $(\psi_n)_n$ be a bounded sequence in $H_{x,y}^1$. Then up to a subsequence, there exist nonzero linear profiles
$(\tdu^j)_j\subset L_x^2 H_y^1$, remainders $(w_n^k)_{k,n}\subset L_x^2 H_y^1$, parameters $(t^j_n,x^j_n,\xi^j_n,\ld^j_n)_{j,n}\subset\R\times\R^2\times\R^2\times(0,\infty)$ and $K^*\in\N\cup\{\infty\}$, such that
\begin{itemize}
\item[(i)] For any finite $1\leq j\leq K^*$ the parameters satisfy
\begin{align}
1&\gtrsim_j\lim_{n\to\infty}|\xi_n^j|,\nonumber\\
\lim_{n\to\infty}t^j_n&=:t_\infty^j\in\{0,\pm\infty\},\nonumber\\
\lim_{n\to\infty}\ld^j_n&=:\ld_\infty^j\in\{1,\infty\},\nonumber\\
t_n^j&\equiv 0\quad\text{if $t_\infty^j=0$},\nonumber\\
\ld_n^j&\equiv 1\quad\text{if $\ld_\infty^j=1$},\nonumber\\
\xi_n^j&\equiv 0\quad\text{if $\ld_\infty^j=1$}.
\end{align}

\item[(ii)]For any finite $1\leq k\leq K^*$ we have the decomposition
\begin{align}\label{cnls decomp}
\psi_n=\sum_{j=1}^k T^j_nP_n^j \tdu^j+w_n^k.
\end{align}
Here, the operators $T_n^j$ and $P_n^j$ are defined by
\begin{align}
T^j_n u(x):=
\left\{
             \begin{array}{ll}
             [e^{it^j_n\Delta_x}u](x-x^j_n,y),&\text{if $\ld^j_\infty=1$},\\
             \\
             g_{\xi^j_n,x^j_n,\ld^j_n}[e^{it^j_n\Delta_x}u](x,y),&\text{if $\ld^j_\infty=\infty$}
             \end{array}
\right.
\end{align}
and
\begin{align}
P^j_n u:=
\left\{
             \begin{array}{ll}
             u,&\text{if $\ld^j_\infty=1$},\\
             \\
             P_{\leq(\ld_n^j)^\theta}u,&\text{if $\ld^j_\infty=\infty$}
             \end{array}
\right.
\end{align}
for some $\theta\in(0,1)$. Moreover,
\begin{align}
\tdu^j\in
\left\{
             \begin{array}{ll}
             H_{x,y}^1,&\text{if $\ld^j_\infty=1$},\\
             \\
             L_x^2 H_y^1,&\text{if $\ld^j_\infty=\infty$}.
             \end{array}
\right.
\end{align}

\item[(iii)] The remainders $(w_n^k)_{k,n}$ satisfy
\begin{align}\label{cnls to zero wnk}
\lim_{k\to K^*}\lim_{n\to\infty}\|e^{it\Delta_x}w_n^k\|_{L_{t,x,y}^4(\R)}=0.
\end{align}

\item[(iv)] The parameters are orthogonal in the sense that
\begin{align}\label{cnls orthog of pairs}
 \frac{\ld_n^k}{\ld_n^j}+ \frac{\ld_n^j}{\ld_n^k}+\ld_n^k|\xi_n^j-\xi_n^k|+\bg|t_k\bg(\frac{\ld_n^k}{\ld_n^j}\bg)^2-t_n^j\bg|
+\bg|\frac{x_n^j-x_n^k-2t_n^k(\ld_n^k)^2(\xi_n^j-\xi_n^k)}{\ld_n^k}\bg|\to\infty
\end{align}
for any $j\neq k$.

\item[(v)] For any finite $1\leq k\leq K^*$ we have the energy decompositions
\begin{align}
\|\psi_n\|_{L_{x,y}^2}^2&=\sum_{j=1}^k\|T_n^jP_n^j\tdu^j\|_{L_{x,y}^2}^2+\| w_n^k\|_{L_{x,y}^2}^2+o_n(1),\label{orthog L2}\\
\|\nabla_x\psi_n\|_{L_{x,y}^2}^2&=\sum_{j=1}^k\|\nabla_xT_n^jP_n^j\tdu^j\|_{L_{x,y}^2}^2+\|\nabla_xw_n^k\|_{L_{x,y}^2}^2+o_n(1),\label{orthog gradx}\\
\|\nabla_y\psi_n\|_{L_{x,y}^2}^2&=\sum_{j=1}^k\|\nabla_yT_n^jP_n^j\tdu^j\|_{L_{x,y}^2}^2+\|\nabla_yw_n^k\|_{L_{x,y}^2}^2+o_n(1),\label{orthog grady}\\
\|\psi_n\|_{L_{x,y}^4}^4&=\sum_{j=1}^k\|T_n^jP_n^j\tdu^j\|_{L_{x,y}^4}^4+\|w_n^k\|_{L_{x,y}^4}^4+o_n(1)\label{cnls conv of h}.
\end{align}
\end{itemize}
\end{lemma}

\begin{proof}
We construct the linear profiles iteratively and start with $k=0$ and $w_n^0:=\psi_n$. We assume initially that the linear profile decomposition is given and its claimed properties are satisfied for some $k$. Define
\begin{align*}
\vare_{k}:=\lim_{n\to\infty}\|e^{it\Delta_x}w_n^k\|_{L_{t,x,y}^4(\R)}.
\end{align*}
If $\vare_k=0$, then we stop and set $K^*=k$. Otherwise we apply Lemma \ref{refined l2 lemma 1} to $w_n^k$ to obtain the sequence $(\tdu^{k+1},w_n^{k+1},t_n^{k+1},x_n^{k+1},\xi_n^{k+1},\ld_n^{k+1})_{n}.$ We should still need to check that the items (iii) and (iv) are satisfied for $k+1$. That the other items are also satisfied for $k+1$ follows directly from the construction of the linear profile decomposition. If $\vare_k=0$, then item (iii) is automatic; otherwise we have $K^*=\infty$. Using \eqref{cnls l2 refined strichartz decomp 1}, \eqref{orthog L2}, \eqref{orthog gradx} and \eqref{orthog grady} we obtain that
\begin{align}
\sum_{j\in N}A_{j-1}^2\bg(\frac{\vare_{j-1}}{A_{j-1}}\bg)^{24}
\lesssim \sum_{j\in \N}\|\tdu^j\|^2_{L_x^2 H_y^1}
=\sum_{j\in \N}\lim_{n\to\infty}\|T_n^jP_n^j \phi^j\|^2_{L_x^2H_y^1}
\leq \lim_{n\to\infty}\|\psi_n\|^2_{L_x^2 H_y^1}= A_0^2,
\end{align}
where $A_j:=\lim_{n\to\infty}\|w_n^j\|_{L_x^2 H_y^1}$. By \eqref{orthog L2} and \eqref{orthog grady} we know that $(A_j)_j$ is monotone decreasing, thus also bounded. Hence
\begin{align*}
A_j^2\bg(\frac{\vare_j}{A_j}\bg)^{24}\to 0\quad\text{as $j\to\infty$}.
\end{align*}
Combining with the boundedness of $(A_j)_j$ we immediately conclude that $\vare_i\to 0$ and the proof of item (iii) is complete. Finally we show item (iv). Assume that item (iv) does not hold for some $j<k$. By the construction of the profile decomposition we have
\begin{align*}
w_n^{k-1}=w_n^j-\sum_{l=j+1}^{k-1}g_n^l e^{-it_n^l\Delta_x}P_n^l \tdu^l.
\end{align*}
Then by definition of $\tdu^k$ we know that
\begin{align}
\tdu^k&=\wlim_{n\to\infty}e^{-it_n^k\Delta}[(g_n^k)^{-1}w_n^{k-1}]\nonumber\\
&=\wlim_{n\to\infty}e^{-it_n^k\Delta}[(g_n^j)^{-1}w_n^{j}]-\sum_{l=j+1}^{k-1}\wlim_{n\to\infty}e^{-it_n^k\Delta}[(g_n^k)^{-1}P_n^l \tdu^l],
\end{align}
where the weak limits are taken in the $L_x^2H_y^1$-topology. We aim to show $\tdu^k$ is zero, which leads to a contradiction and proves item (iv). For the first summand, we obtain that
\begin{align*}
e^{-it_n^k\Delta_x}[(g_n^k)^{-1}w_n^{j}]=(e^{-it_n^k\Delta_x}(g_n^k)^{-1}g_n^je^{it_n^j\Delta_x})[e^{-it_n^j\Delta_x}(g_n^j)^{-1}w_n^j].
\end{align*}
Direct calculation yields
\begin{align}\label{cnls composite of g's}
&e^{-it_n^k\Delta}(g_n^k)^{-1}g_n^je^{it_n^j\Delta_x}\nonumber\\
=&\,\beta_{n}^{j,k}g_{\ld_n^k(\xi_n^j-\xi_n^k),\frac{x_n^j-x_n^k-2t_n^k(\ld_n^k)^2(\xi_n^j-\xi_n^k)}{\ld_n^k},\frac{\ld_n^j}{\ld_n^k}}
e^{-i\bg(t_n^k\bg(\frac{\ld_n^k}{\ld_n^j}\bg)^2-t_n^j\bg)\Delta_x}.
\end{align}
with $\beta_{n}^{j,k}=e^{i(\xi_n^j-\xi_n^k)x_n^k+t_n^k(\ld_n^k)^2|\xi_n^j-\xi_n^k|^2}$. Therefore, the failure of item (iv) will lead to the strong convergence of the adjoint of $e^{-it_n^k\Delta}(g_n^k)^{-1}g_n^je^{it_n^j\Delta_x}$ on $L_x^2 H_y^1$. By construction of the profile decomposition we have
\begin{align*}
e^{-it_n^j\Delta}(g_n^j)^{-1}w_n^j\rightharpoonup 0\quad\text{in $L_x^2 H_y^1$}
\end{align*}
and we conclude that the first summand weakly converges to zero in $L_x^2 H_y^1$. Now we consider the single terms in the second summand. We can rewrite each single summand to
\begin{align*}
e^{-it_n^k\Delta_x}[(g_n^k)^{-1}P_n^l \tdu^l]=(e^{-it_n^k\Delta_x}(g_n^k)^{-1}g_n^je^{it_n^j\Delta_x})[e^{-it_n^j\Delta_x}(g_n^j)^{-1}P_n^l \tdu^l].
\end{align*}
By the previous arguments it suffices to show that
\begin{align*}
e^{-it_n^j\Delta_x}(g_n^j)^{-1}P_n^l \tdu^l\rightharpoonup 0\quad\text{in $L_x^2 H_y^1$}.
\end{align*}
Due to the construction of the profile decomposition and the inductive hypothesis we know that $\tdu^l\in L_x^2H_y^1$ and item (iv) is satisfied for the pair $(j,l)$. Using the fact that
\begin{align*}
\|P_{\leq(\ld_n^l)^\theta}\tdu^l-\tdu^l\|_{L_x^2 H_y^1}\to 0\quad\text{when $\ld_n^l\to \infty$}
\end{align*}
and density arguments, it suffices to show that
\begin{align*}
I_n:=e^{-it_n^j\Delta}(g_n^j)^{-1}g_n^l e^{it_n^l\Delta_x} \tdu\rightharpoonup 0\quad\text{in $L_x^2 H_y^1$}
\end{align*}
for arbitrary $\tdu\in C_c^\infty(\R^2)\otimes C_{\mathrm{per}}^\infty(\T)$. Using \eqref{cnls composite of g's} we obtain that
\begin{align*}
I_n=\beta_n^{j,l}g_{\ld_n^l(\xi_n^j-\xi_n^l),\frac{x_n^j-x_n^l-2t_n^l(\ld_n^l)^2(\xi_n^j-\xi_n^l)}{\ld_n^l},\frac{\ld_n^j}{\ld_n^l}}
e^{-i\bg(t^l_n\bg(\frac{\ld_n^l}{\ld_n^j}\bg)^2-t_n^j\bg)\Delta_x}\tdu.
\end{align*}
Assume first that $\lim_{n\to\infty}\frac{\ld_n^j}{\ld_n^l}+\frac{\ld_n^l}{\ld_n^j}=\infty$. Then for any $\psi\in C_c^\infty(\R^2)\otimes
C^\infty_{\mathrm{per}}(\T)$ we have
\begin{align*}
|\la I_n,\psi\ra_{L_x^2 H_y^1}|\leq \min\bg\{\bg(\frac{\ld_n^j}{\ld_n^l}\bg)^{-1}\|\mathcal{F}_x\phi\|_{L_x^1 H_y^1}\|\mathcal{F}_x\psi\|_{L_x^\infty H_y^1},\,
\bg(\frac{\ld_n^l}{\ld_n^j}\bg)^{-1}\|\mathcal{F}_x\psi\|_{L_x^1 H_y^1}\|\mathcal{F}_x\phi\|_{L_x^\infty H_y^1}\bg\}\to 0.
\end{align*}
So we may assume that $\lim_{n\to\infty}\frac{\ld_n^j}{\ld_n^l}\in(0,\infty)$. Suppose now $t^l_n\bg(\frac{\ld_n^l}{\ld_n^j}\bg)^2-t_n^j\to\pm\infty$. Then the weak convergence of $I_n$ to zero in $L_x^2 H_y^1$ follows immediately from the dispersive estimate. Hence we may also assume that $\lim_{n\to\infty}t^l_n\bg(\frac{\ld_n^l}{\ld_n^j}\bg)^2-t_n^j\in\R$. Finally, it is left with the options
\begin{align*}
|\ld_n^l(\xi_n^j-\xi_n^l)|\to\infty\quad\text{or}\quad\bg|\frac{x_n^j-x_n^l-2t_n^l(\ld_n^l)^2(\xi_n^j-\xi_n^l)}{\ld_n^l}\bg|\to\infty.
\end{align*}
In the latter case, we utilize the fact that the symmetry group composing by unbounded translations in $L_x^2$ weakly converges to zero as operators in $L_x^2H_y^1$ to deduce the claim; In the former case, we can use the same arguments as the ones for the translation symmetry by considering the Fourier transformation of $I_n$ (w.r.t. $x$) in the frequency space. This completes the desired proof of item (iv).
\end{proof}

\begin{remark}
By interpolation and Strichartz we have for $s\in(\frac12,1)$
\begin{align}
\lim_{k\to K^*}\lim_{n\to\infty}\|e^{it\Delta_{x,y}}w_n\|_{\diag H_y^{s}(\R)}
&=\lim_{k\to K^*}\lim_{n\to\infty}\|e^{it\Delta_{x}}w_n\|_{\diag H_y^{s}(\R)}\nonumber\\
&\lesssim\lim_{k\to K^*}\lim_{n\to\infty}\|e^{it\Delta_{x}}w_n\|^{1-s}_{\diag L_y^2(\R)}\|w_n\|^s_{H_{x,y}^1}\nonumber\\
&\lesssim\lim_{k\to K^*}\lim_{n\to\infty}\|e^{it\Delta_{x}}w_n\|^{1-s}_{L_{t,x,y}^4(\R)}=0.\label{interpolation remainder}
\end{align}
\end{remark}

\subsection{Large scale approximation}
The following lemma shows that large scale nonlinear profiles can be well approximated by the large scale resonant system \eqref{nls}.
\begin{lemma}[Large scale approximation]\label{cnls lem large scale proxy}
Let $(\ld_n)_n\subset(0,\infty)$ such that $\ld_n\to \infty$, $(t_n)_n\subset\R$ such that either $t_n\equiv 0$ or $t_n\to\pm\infty$ and $(\xi_n)_n\subset\R^d$ such that $(\xi_n)_n$ is bounded. Let $\phi\in L_{x}^2 H_y^1$ and define
$$\phi_n:=g_{\xi_n,x_n,\ld_n}e^{it_n\Delta}P_{\leq \ld_n^\theta}\tdu$$
for some $\theta\in(0,1)$. Assume also that $\mM(\phi)<\pi\mM(Q_{2d})$. Then for all sufficiently large $n$ the solution $u_n$ of \eqref{cnls} with $U_n(0)=\phi_n$ is global and scattering in time with
\begin{align}
\limsup_{n\to\infty}\|U_n\|_{\diag H_y^1(\R)}&\leq C(\|\phi\|_{L_x^2 H_y^1}).\label{L2 proxy 1}
\end{align}
Furthermore, for every $\beta>0$ there exists $N_\beta\in\N$ and $\psi_\beta\in C_c^\infty(\R\times \R^2)\otimes C_{\mathrm{per}}^\infty(\T)$ such that
\begin{align}
\bg\|U_n-\ld_n^{-1}e^{-it|\xi_n|^2}e^{i\xi_n\cdot x}\psi_\beta\bg(\frac{t}{\ld_n^2}+t_n,\frac{x-x_n-2t\xi_n}{\ld_n},y\bg)\bg\|_{\diag H_y^{1}(\R)}\leq \beta\label{L2 proxy 3}
\end{align}
for all $n\geq N_\beta$.
\end{lemma}

\begin{proof}
The proof is almost identical to the one of \cite[Lem. 3.11]{CubicR2T1Scattering}, we only need to replace the large data scattering result \cite[Thm. 1.1]{Yang_Zhao_2018} therein for the defocusing analogue of \eqref{nls} to Theorem \ref{thm large scale resonant} (hence we also impose the mass restriction $\mM(\phi)<\pi\mM(Q_{2d})$). We omit therefore the repeating arguments.
\end{proof}

\begin{remark}
We explain where the prefactor $\pi$ comes from. By our definition of the Fourier series, for $\phi\in H_{x,y}^1$ we have the Fourier inverse formula
\begin{align}
\phi(x,y)=\sum_{k\in \Z}(2\pi)^{-\frac12}e^{iky}\mathcal{F}_y\phi(x,k)
\end{align}
and the Plancherel's isometry formula
\begin{align}
\|\phi\|^2_{L_{x,y}^2}=\|\mathcal{F}_y\phi\|_{\ell^2 L_x^2}^2.
\end{align}
The initial data $V_\phi$ for the large scale proxy is defined by
$$V_\phi:=(e^{iky}(2\pi)^{-\frac12}\mathcal{F}_y\phi(x,k))_{k\in\Z}.$$
In order to apply Theorem \ref{thm large scale resonant}, we then demand
$$ (2\pi)^{-1}\|\phi\|^2_{L_{x,y}^2}=(2\pi)^{-1}\|\mathcal{F}_y\phi\|_{\ell^2 L_x^2}^2<2^{-1}\mM(Q_{2d}).$$
\end{remark}

\subsection{Variational analysis}\label{subsec: r2t1 variational}
We begin with the proof of Proposition \ref{prop gn cnls}.

\begin{proof}[Proof of Proposition \ref{prop gn cnls}]
For a function $u$ we define $m(u):=(2\pi)^{-1}\int_{\T}u(y)\,dy$. Then
$$\|u\|_{L_{x,y}^4}\leq \|m(u)\|_{L_{x,y}^4}+\|u-m(u)\|_{L_{x,y}^4}.$$
We will show that $\|m(u)\|_{L_{x,y}^4}$ and $\|u-m(u)\|_{L_{x,y}^4}$ are bounded by the first term and second term of \eqref{r2t1gn} respectively, which will complete the proof. For $\|m(u)\|_{L_{x,y}^4}$, we use \eqref{standard gn ineq} and Jensen to infer
\begin{align}
\|m(u)\|^4_{L_{x,y}^4}&=2\pi\|m(u)\|^4_{L_{x}^4}\nonumber\\
&\leq 2\pi\mathrm{C}_{\mathrm{GN},2d}^{-1}\|m(u)\|_{L_x^2}^2\|\nabla_x m(u)\|_{L_x^2}^2\nonumber\\
&\leq2\pi\mathrm{C}_{\mathrm{GN},2d}^{-1}(2\pi)^{-1}(2\pi)^{-1}\|u\|_{L_{x,y}^2}^2\|\nabla_x u\|_{L_{x,y}^2}^2
=(\pi\mM(Q_{2d}))^{-1}\|u\|_{L_{x,y}^2}^2\|\nabla_x u\|_{L_{x,y}^2}^2.
\end{align}
To estimate the second term, we recall the Sobolev inequality on torus for functions with zero mean (see for instance \cite{sobolev_torus}): for $s>0$ and $1<p<q<\infty$ with $s/d=1/p-1/q$ we have
\begin{align}
\|u\|_{L^q(\T^d)}\lesssim \|u\|_{\dot{L}^p_s(\T^d)}.
\end{align}
Therefore, setting $s=\frac12$, $p=2$ and $q=4$ we obtain
\begin{align}
\|u\|^4_{L_y^4}\lesssim \|u\|^4_{\dot{H}_y^{\frac14}}\lesssim \|u\|_{L_y^2}^3\|\nabla_y u\|_{L^2_y}
\end{align}
and the existence of the number $\mathrm{C}_\T$ from Proposition \ref{prop gn cnls} follows. Using H\"older we conclude
\begin{align}
\|u-m(u)\|^4_{L_{x,y}^4}\leq \mathrm{C}_{\T}\|u-m(u)\|_{L_x^6L_y^2}^3\|\nabla_y u\|_{L^2_{x,y}}.
\end{align}
Followed by Minkowski, Gagliardo-Nirenberg (bounding $L_x^6$ by $L_x^2-\dot{H}_x^1$ in 2D), H\"older and Jensen, we see that
\begin{align}
\|u-m(u)\|_{L_x^6L_y^2}^3&\leq \|u-m(u)\|_{L_y^2 L_x^6}^3\nonumber\\
&\leq\widehat{\rm G}_{\mathrm{GN},2d}^{-\frac12}\|\|u-m(u)\|_{L_x^2}^{\frac13}\|\nabla_x(u-m(u))\|_{L_x^2}^{\frac23}\|_{L_y^2}^3\nonumber\\
&\leq\widehat{\rm G}_{\mathrm{GN},2d}^{-\frac12}\|u-m(u)\|_{L_{x,y}^2}\|\nabla_x(u-m(u))\|_{L_{x,y}^2}^{2}\nonumber\\
&\leq \widehat{\rm G}_{\mathrm{GN},2d}^{-\frac12}(1+(2\pi)^{-\frac12})^3\|u\|_{L_{x,y}^2}\|\nabla_x u\|_{L_{x,y}^2}^{2},
\end{align}
which completes the desired proof.
\end{proof}

We next prove a crucial energy trapping result based on Proposition \ref{prop gn cnls}.

\begin{lemma}[Energy trapping]\label{holmer variational}
Let $c_*$ and $\Gamma$ be defined through \eqref{defcdelta} and \eqref{gamma def} respectively. Suppose that $U_0\in H_{x,y}^1$ satisfies \eqref{weaker thres}, \eqref{threshold2} and \eqref{threshold3}. Let $U$ be the solution of \eqref{cnls} with $U(0)=U_0$. Then for all $t\in I_{\max}$, where $I_{\max}$ is the maximal lifespan of $U$, we have
\begin{align}
\|\nabla_y U(t)\|_{L^2_{x,y}}^2<
\Gamma(U_0).\label{threshold4}
\end{align}
Moreover, if there exists some $\beta\in(0,1)$ such that
\begin{align}
\mM(U_0)&\leq(1-\beta)2\pi\mM(Q_{2d}),\label{threshold7}\\
\mH(U_0)&\leq(1-\beta)2^{-1}\Gamma(U_0),\label{threshold8}
\end{align}
then
\begin{align}
\|\nabla_{x,y}U(t)\|^2_{L_{x,y}^2}&\lesssim_\beta \mH(U_0),\label{threshold5}\\
\|\nabla_{x}U(t)\|^2_{L_{x,y}^2}&\lesssim_\beta \frac{1}{2}\|\nabla_x U(t)\|_{L_{x,y}^2}^2-\frac{1}{4}\|U(t)\|_{L_{x,y}^4}^4
=:\mH_*(U(t)).\label{thres no energy}
\end{align}
\end{lemma}

\begin{proof}
Using \eqref{r2t1gn} we have
\begin{align}
\mH(U_0)&\geq \frac12\|\nabla_y U(t)\|_{L_{x,y}^2}^2
+\frac{1}{4}\|\nabla_x U(t)\|_{L_{x,y}^2}^2\bg(2-\bg((\pi\mM(Q_{2d}))^{-\frac14}\mM(U)^{\frac14}+c_*\mM(U)^{\frac18}\|\nabla_y U(t)\|_{L_{x,y}^2}^{\frac14}\bg)^4\bg)\nonumber\\
&=:\frac12\|\nabla_y U(t)\|_{L_{x,y}^2}^2+\frac{1}{4}\|\nabla_x U(t)\|_{L_{x,y}^2}^2\Xi(\|\nabla_y U(t)\|_{L_{x,y}^2}^2).
\label{threshold6}
\end{align}
Similarly,
\begin{align}
\mH_*(U(t))\geq  \frac{1}{4}\|\nabla_x U(t)\|_{L_{x,y}^2}^2\Xi(\|\nabla_y U(t)\|_{L_{x,y}^2}^2).\label{thres no energy 2}
\end{align}
One easily verifies that $\Gamma(U)$ is a root of $\Xi$. Since $\|\nabla_y U_0\|_{L_{x,y}^2}^2<\Gamma(U)$, if there exists some $t\in I_{\max}$ such that $\|\nabla_y U(t)\|_{L_{x,y}^2}^2\geq\Gamma(U)$, then by continuity there exists some $s\in(0,t]$ such that
$\|\nabla_y U(s)\|_{L_{x,y}^2}^2=\Gamma(U)$. But then we obtain the contradiction
\begin{align}
2^{-1}\Gamma(U)>\mH(U_0)=\mH(U(s))=2^{-1}\Gamma(U),
\end{align}
which implies \eqref{threshold4}. Next, we take \eqref{threshold5} and \eqref{thres no energy}. \eqref{threshold8} implies $\|\nabla_y U(t)\|_{L_{x,y}^2}^2\leq(1-\beta)\Gamma(U)$, which in turn implies
\begin{align}
2^{\frac14}-\bg((\pi\mM(Q_{2d}))^{-\frac14}\mM(U)^{\frac14}+c_*\mM(U)^{\frac18}\|\nabla_y U(t)\|_{L_{x,y}^2}^{\frac14}\bg)
\geq (1-(1-\beta)^{\frac18})(1-(1-\beta)^{\frac14})2^{\frac14}>0,
\end{align}
which combining with \eqref{threshold6} and \eqref{thres no energy 2} implies \eqref{threshold5} and \eqref{thres no energy}.
\end{proof}

At the end of this section, we introduce the MEI-functional $\mD$ which plays a fundamental role for setting up the inductive hypothesis. Such functional was firstly introduced in \cite{killip_visan_soliton} and is quite useful for building up a multi-directional inductive hypothesis scheme. Define the domain $\Omega\subset \R^2$ by
\begin{align}
\Omega&:=\bg((-\infty,0]\times \R\bg)\cup\bg\{(c,h)\in\R^2:c\in(0,\pi\mM(Q_{2d})),h\in(-\infty,2^{-1}\Gamma(c))\bg\}.
\end{align}
Then we define the MEI-functional $\mD:\R^2\to [0,\infty]$ by
\begin{align}\label{cnls MEI functional}
\mD(c,h)=\left\{
             \begin{array}{ll}
             h+c(\pi\mM(Q_{2d})-c)^{-1}+h(2^{-1}\Gamma(c)-h)^{-1},&\text{if $(c,h)\in \Omega$},\\
             \infty,&\text{otherwise}.
             \end{array}
\right.
\end{align}
For $U\in H_{x,y}^1$, define $\mathcal{D}(U):=\mathcal{D}(\mM(U),\mH(U))$. Also define the quantity $\mK(U):=\Gamma(U)-\|\nabla_y U\|^2_{L_{x,y}^2}$ and the set $\mA$ by
\begin{align}
\mA&:=\{U\in H_{x,y}^1:0<\mM(U)<\pi\mM(Q_{2d}),\,\mH(U)<2^{-1}\Gamma(U),\,\mK(U)>0\}.
\end{align}
By conservation of mass and energy we know that if $U$ is a solution of \eqref{cnls}, then $\mD(U(t))$ is a conserved quantity, thus in the following we simply write $\mD(U)=\mD(U(t))$. Moreover, by Lemma \ref{holmer variational} we know that if $U(t)\in \mA$ for some $t$ in the lifespan of $U$, then $U(t)\in\mA$ for all $t$ in the maximal lifespan of $U$. In this case, we simply write $U\in\mA$.

We end this section by giving some useful properties of the MEI-functional.

\begin{lemma}\label{cnls killip visan curve}
Let $U,U_1,U_2$ be solutions of \eqref{cnls}. The following statements hold:
\begin{itemize}
\item[(i)] Let $\mK(U(t))>0$ for some $t$ in the lifespan of $U$. Then $0<\mD(U)<\infty$ if and only if $U\in\mA$.
\item[(ii)] Let $U_1,U_2\in \mA$ satisfy $\mM(U_1)\leq \mM(U_2)$ and $\mH(U_1)\leq \mH(U_2)$, then $\mD(U_1)\leq \mD(U_2)$. If in addition either $\mM(U_1)<\mM(U_2)$ or $\mH(U_1)<\mH(U_2)$, then $\mD(U_1)<\mD(U_2)$.
\item[(iii)] Let $\mD_0\in(0,\infty)$. Then
\begin{align}
\|\nabla_{x,y}U\|^2_{L_{x,y}^2}&\lesssim_{\mD_0}\mH(U),\label{mei var1}\\
\|U\|^2_{H_{x,y}^1}&\lesssim_{\mD_0}\mH(U)+\mM(U)\lesssim_{\mD_0}\mD(U)\label{mei var2}
\end{align}
uniformly for all $U\in \mA$ with $\mD(U)\leq \mD_0$.
\end{itemize}
\end{lemma}

\begin{proof}
(i) follows immediately from Lemma \ref{holmer variational}. (ii) follows directly from the definition of $\mD$ and $\Gamma(U)$. Now we take (iii). Since $U\in\mA$, we know that $\mM(U)\in(0,\pi\mM(Q_{2d}))$ and using Lemma \ref{holmer variational} also $\mH(U)\in [0,2^{-1}\Gamma(U))$. By definition of $\mD_0$ we infer that
\begin{align}\label{mass bound}
\mD_0\geq\mD(U)\geq \frac{\mM(U)}{\pi\mM(Q_{2d})-\mM(U)},
\end{align}
which in turn implies
\begin{align}
\mM(U)\leq \frac{\mD_0}{1+\mD_0}\pi\mM(Q_{2d}).
\end{align}
Similarly, we deduce
\begin{align}
\mH(U)\leq \frac{\mD_0}{1+\mD_0}2^{-1}\Gamma(U)
\end{align}
and \eqref{mei var1} follows from Lemma \ref{holmer variational}. The first inequality of \eqref{mei var2} follows already from \eqref{mei var1}. Next, we obtain that \eqref{mass bound} also implies
\begin{align}\label{mass upper bd}
\mM(U)\leq\frac{\mD(U)\pi\mM(Q_{2d})}{1+\mD(U)}\leq\mD(U)\pi\mM(Q_{2d}).
\end{align}
Together with $\mD(U)\geq\mH(U)$, which is deduced directly form the definition of $\mD$, the desired claim follows.
\end{proof}

\subsection{Existence of a minimal blow-up solution}
Having all the preliminaries we are ready to construct a minimal blow-up solution of \eqref{cnls}. For convenience, we simply fix the number $s$ in Lemma \ref{lemma cnls well posedness} to $s=\frac23$. This number can be replaced by any number from the interval $(\frac12,1)$, but we need to restrict the number to be smaller than one in order to apply \eqref{interpolation remainder}. Define
\begin{align*}
\tau(\mD_0):=\sup\bg\{\|U\|_{\diag H_y^{1}(I_{\max})}:
\text{ $U$ is solution of \eqref{cnls}, }U(0)\in {\mA},\mD(U)\leq \mD_0\bg\}
\end{align*}
and
\begin{align}\label{introductive hypothesis}
\mD^*&:=\sup\{\mD_0>0:\tau(\mD_0)<\infty\}.
\end{align}
By Lemma \ref{lemma cnls well posedness} and Lemma \ref{cnls killip visan curve} we know that $\mD^*>0$ and $\tau(\mD_0)<\infty$ for sufficiently small $\mD_0$. We will therefore assume $\mD^*<\infty$ and aim to derive a contradiction, which will imply $\mD^*=\infty$ and the proof of Theorem \ref{main thm} will be complete in view of Lemma \ref{cnls killip visan curve}. By the inductive hypothesis we can find a sequence $(U_n)_n$ which are solutions of \eqref{cnls} with $(U_n(0))_n\subset {\mA}$ and maximal lifespan $(I_{n})_n$ such that
\begin{gather}
\lim_{n\to\infty}\|U_n\|_{\diag H_y^{1}((\inf I_n,0])}=\lim_{n\to\infty}\|U_n\|_{\diag H_y^{1}([0, \sup I_n))}=\infty,\label{oo1}\\
\lim_{n\to\infty}\mD(U_n)=\mD^*.\label{oo2}
\end{gather}
Up to a subsequence we may also assume that
\begin{align*}
(\mM(U_n),\mH(U_n))\to(\mM_0,\mH_0)\quad\text{as $n\to\infty$}.
\end{align*}
By continuity of $\mD$ and finiteness of $\mD^*$ we know that
\begin{align*}
\mD^*=\mD(\mM_0,\mH_0),\quad
\mM_0\in(0,\pi\mM(Q_{2d})),\quad
\mH_0\in[0,2^{-1}\Gamma(U)).
\end{align*}
From Lemma \ref{cnls killip visan curve} it follows that $(U_n(0))_n$ is a bounded sequence in $H_{x,y}^1$, hence Lemma \ref{linear profile} is applicable for $(U_n(0))_n$. We define the nonlinear profiles as follows: For $\ld_\infty^k=\infty$, we define $U_n^k$ as the solution of \eqref{cnls} with $U_n^k(0)=T_n^kP_n^k\tdu^k$. For $\ld_\infty^k=1$ and $t^k_\infty=0$, we define $U^k$ as the solution of \eqref{cnls} with $U^k(0)=\tdu^k$; For $\ld_\infty^k=1$ and $t^k_\infty\to\pm\infty$, we define $U^k$ as the solution of \eqref{cnls} that scatters forward (backward) to $e^{it\Delta_x}\tdu^k$ in $H_{x,y}^1$. In both cases for $\ld_\infty^k=1$ we define
\begin{align*}
U_n^k:=U^j(t+t_n,x-x_n^k,y).
\end{align*}
Then $U_n^j$ is also a solution of \eqref{cnls}. In all cases we have for each finite $1\leq k \leq K^*$
\begin{align}\label{conv of nonlinear profiles in h1}
\lim_{n\to\infty}\|U_n^k(0)-T_n^kP_n^k\tdu^k\|_{H_{x,y}^1}=0.
\end{align}

In the following, we establish a Palais-Smale type lemma which is essential for the construction of the minimal blow-up solution.

\begin{lemma}[Palais-Smale-condition]\label{Palais Smale}
Let $(U_n)_n$ be a sequence of solutions of \eqref{cnls} with maximal lifespan $I_n$, $U_n\in\mA$ and $\lim_{n\to\infty}\mD(U_n)=\mD^*$. Assume also that there exists a sequence $(t_n)_n\subset\prod_n I_n$ such that
\begin{align}\label{precondition}
\lim_{n\to\infty}\|U_n\|_{\diag H_y^{1}((\inf I_n,\,t_n])}=\lim_{n\to\infty}\|U_n\|_{\diag H_y^{1}([t_n,\,\sup I_n)}=\infty.
\end{align}
Then up to a subsequence, there exists a sequence $(x_n)_n\subset\R^2$ such that $(U_n(t_n, \cdot+x_n,y))_n$ strongly converges in $H_{x,y}^1$.
\end{lemma}

\begin{proof}
By time translation invariance we may assume that $t_n\equiv 0$. Let $(U_n^j)_{j,n}$ be the nonlinear profiles corresponding to the linear profile decomposition of $(U_n(0))_n$. Define
\begin{align*}
\Psi_n^k:=\sum_{j=1}^k U_n^j+e^{it\Delta_{x,y}}w_n^k.
\end{align*}
We will show that there exists exactly one non-trivial bad linear profile, relying on which the desired claim follows. We divide the remaining proof into three steps.
\subsubsection*{Step 1: Positive energies of the linear profiles}
Since the nonlinearity is focusing, it is \textit{a priori} unclear whether the linear profiles have non-negative energies. We show that this is indeed the case for sufficiently large $n$. Using \eqref{orthog L2} to \eqref{cnls conv of h} we conclude that for any finite $1\leq k\leq K^*$
\begin{align}
\mM_0&=\sum_{j=1}^k \mM(T_n^jP_n^j\tdu^j)+\mM(w_n^k)+o_n(1),\label{mo sum}\\
\mH_0&=\sum_{j=1}^k \mH(T_n^jP_n^j\tdu^j)+\mH(w_n^k)+o_n(1)\label{eo sum}\\
\|\nabla_yU_n(0)\|^2_{L_{x,y}^2}&=
\sum_{j=1}^k \|\nabla_yT_n^jP_n^j\tdu^j\|^2_{L_{x,y}^2}
+\|\nabla_yw_n^k\|_{L_{x,y}^2}+o_n(1)\label{ko sum}.
\end{align}
By \eqref{ko sum} and the fact that $U_n(0)\in\mA$ we know that for given $1\leq k\leq K^*$ we have $\mK(T_n^kP_n^k\tdu^k)>0$ and $\mK(w_n^k)>0$ for sufficiently large $n$. If in this case $\mH(T_n^kP_n^k\tdu^k)$ were negative, then
$$\mH(T_n^kP_n^k\tdu^k)<0\leq 2^{-1}\Gamma(T_n^kP_n^k\tdu^k),$$
which contradicts Lemma \ref{holmer variational} and we conclude that $\mH(T_n^kP_n^k\tdu^k)\geq 0$ for given $1\leq k\leq K^*$ and all $n\geq N_1$ for some large $N_1=N_1(k)$. The same holds for $w_n^k$ and the proof of Step 1 is complete.
\subsubsection*{Step 2: Decoupling of nonlinear profiles}
In this step, we show that the nonlinear profiles are asymptotically decoupled in the sense that
\begin{align}\label{smallness aaa}
\lim_{n\to\infty} \|U_n^i U_n^j\|_{L^2_{t,x}W_y^{1,1}(\R)}=\lim_{n\to\infty} \sum_{s_1,s_2=0}^1\|\pt_y^{s_1}U_n^i \pt_y^{s_2} U_n^j\|_{L^2_{t,x}L_y^{1}(\R)}=0
\end{align}
for any fixed $1\leq i,j\leq K^*$ with $i\neq j$, provided that
$$\limsup_{n\to\infty}\,(\|U_n^i\|_{\diag H_y^{1}(\R)}+\|U_n^j\|_{\diag H_y^{1}(\R)})<\infty.$$
We claim that for any $\beta>0$ there exists some $\psi^i_\beta,\psi_\beta^j\in C_c^\infty(\R\times\R^2)\otimes C_{\mathrm{per}}^\infty(\T)$ such that
\begin{align}
\bg\|U^i_n-(\ld^i_n)^{-1}e^{-it|\xi^i_n|^2}e^{i\xi^i_n\cdot x}\psi^i_\beta\bg(\frac{t}{(\ld^i_n)^2}+t^i_n,\frac{x-x^i_n-2t\xi^i_n}{\ld^i_n},y\bg)\bg\|_{\diag H_y^{1}(\R)}\leq \beta,\\
\bg\|U^j_n-(\ld^j_n)^{-1}e^{-it|\xi^j_n|^2}e^{i\xi^j_n\cdot x}\psi^j_\beta\bg(\frac{t}{(\ld^j_n)^2}+t^j_n,\frac{x-x^j_n-2t\xi^j_n}{\ld^j_n},y\bg)\bg\|_{\diag H_y^{1}(\R)}\leq \beta.
\end{align}
Indeed, for $\ld_\infty^i,\ld_\infty^j=\infty$, this follows already from \eqref{L2 proxy 3}, while for $\ld_\infty^i,\ld_\infty^j=1$ we choose some $\psi^i_\beta,\psi^j_\beta\in C_c^\infty(\R\times\R^d)$ such that
\begin{align}
\|U^i-\psi^i_\beta\|_{\diag H_y^{1}(\R)}\leq \beta,\,\|U^j-\psi^j_\beta\|_{\diag H_y^{1}(\R)}\leq \beta.
\end{align}
Define
$$ \Lambda_n (\psi_\beta^i):=(\ld^i_n)^{-1}\psi^i_\beta\bg(\frac{t}{(\ld^i_n)^2}+t^i_n,\frac{x-x^i_n-2t\xi^i_n}{\ld^i_n},y\bg).$$
Using H\"older we infer that
\begin{align*}
\|\pt_y^{s_1} U_n^i \pt_y^{s_2} U_n^j\|_{L_{t,x}^{2}L_y^{1}(\R)}\lesssim \beta+\|\pt_y^{s_1}\Lambda_n (\psi_\beta^i)\pt_y^{s_2}\Lambda_n (\psi_\beta^j)\|_{L_{t,x}^{2}L_y^{1}(\R)}.
\end{align*}
Since $\beta$ can be chosen arbitrarily small, it suffices to show
\begin{align}\label{step2a1}
\lim_{n\to\infty}\|\pt_y^{s_1}\Lambda_n (\psi_\beta^i)\pt_y^{s_2}\Lambda_n (\psi_\beta^j)\|_{L_{t,x}^{2}L_y^1(\R)}=0.
\end{align}
Assume that $\frac{\ld_n^i}{\ld_n^j}+\frac{\ld_n^j}{\ld_n^i}\to \infty$. By symmetry we may w.l.o.g. assume that $\frac{\ld_n^i}{\ld_n^j}\to 0$. Using change of variables we obtain that
\begin{align}
&\,\|\pt_y^{s_1}\Lambda_n (\psi_\beta^i)
\pt_y^{s_2}\Lambda_n (\psi_\beta^j)\|_{L_{t,x}^{2}L_y^1(\R)}\nonumber\\
=&\,\frac{\ld_n^i}{\ld_n^j}\bg\|\pt_y^{s_1}\psi_\beta^i(t,x)\pt_y^{s_2}\psi_\beta^j\bg(
\bg(\frac{\ld_n^i}{\ld_n^j}\bg)^{2}t-\bg(\bg(\frac{\ld_n^i}{\ld_n^j}\bg)^{2}t_n^i-t_n^j,y\bg),
\nonumber\\
&\quad\quad\quad\quad
\bg(\frac{\ld_n^i}{\ld_n^j}\bg)x+2\bg(\frac{\ld_n^i}{\ld_n^j}\bg)\ld_n^i(\xi_n^i-\xi_n^j)t
+\frac{x_n^i-x_n^j-2t_n^i(\ld_n^i)^2(\xi_n^i-\xi_n^j)}{\ld_n^j},y
 \bg)\bg\|_{L_{t,x}^{2}L_y^1(\R)}\label{verylong2}\\
\lesssim&\,\ld_n^i(\ld_n^j)^{-1}\|\psi_\beta^i\|_{L_{t,x}^{2}H_y^{1}(\R)}\|\psi_\beta^j\|_{L_{t,x}^\infty W_y^{1,\infty}(\R)}
\to 0\nonumber
\end{align}
and the claim follows. Suppose therefore $\frac{\ld_n^i}{\ld_n^j}+\frac{\ld_n^j}{\ld_n^i}\to \ld_0\in(0,\infty)$. If $\bg(\frac{\ld_n^i}{\ld_n^j}\bg)^{2}t_n^i-t_n^j\to\pm\infty$, then by \eqref{verylong2} the supports of the integrands become disjoint in the temporal direction. We may therefore further assume that $\bg(\frac{\ld_n^i}{\ld_n^j}\bg)^{2}t_n^i-t_n^j\to t_0\in\R$. If $\bg|\frac{x_n^i-x_n^j-2t_n^i(\ld_n^i)^2(\xi_n^i-\xi_n^j)}{\ld_n^j}\bg|\to\infty$ and $\xi_n^i= \xi_n^j$ for infinitely many $n$, then the supports of the integrands become disjoint in the $x$-spatial direction; If $\bg|\frac{x_n^i-x_n^j-2t_n^i(\ld_n^i)^2(\xi_n^i-\xi_n^j)}{\ld_n^j}\bg|\to\infty$ and $\xi_n^i\neq \xi_n^j$ for infinitely many $n$, then we apply the change of temporal variable $t\mapsto \frac{t}{\ld_n^i|\xi_n^i-\xi_n^j|}$ to see the decoupling of the supports of the integrands in the $x$-spatial direction. Finally, if $\frac{x_n^i-x_n^j-2t_n^i(\ld_n^i)^2(\xi_n^i-\xi_n^j)}{\ld_n^j}\to x_0\in\R^d$, then by \eqref{cnls orthog of pairs} we must have $\ld_n^i|\xi_n^i-\xi_n^j|\to\infty$. Hence for all $t\neq 0$ the integrand converges pointwise to zero. Using the dominated convergence theorem (setting $\|\psi_\beta^j\|_{L_{t,x,y}^\infty(\R)}\psi_\beta^i$ as the majorant) we finally conclude \eqref{step2a1}.

\subsubsection*{Step 3: Existence of at least one bad profile}
First we claim that there exists some $1\leq J\leq K^*$ such that for all $j\geq J+1$ and all sufficiently large $n$, $U_n^j$ is global and
\begin{align}\label{uniform bound of unj}
\sup_{J+1\leq j\leq K^*}\lim_{n\to\infty}\|U_n^j\|_{\diag H_y^{1}(\R)}\lesssim 1.
\end{align}
Indeed, using \eqref{orthog L2} to \eqref{orthog grady} we infer that
\begin{align}\label{small initial data}
\lim_{k\to K^*}\lim_{n\to\infty}\sum_{j=1}^k\|T_n^jP_n^j\tdu^j\|^2_{H_{x,y}^1}<\infty.
\end{align}
Then \eqref{uniform bound of unj} follows from Lemma \ref{lemma cnls well posedness}. In the same manner, by Lemma \ref{lemma cnls well posedness} we infer that
\begin{align}\label{correct1}
\sup_{J+1\leq k\leq K^*}\lim_{n\to\infty} \sum_{j=J+1}^k\|U_n^j\|^2_{\diag H_y^{1}(\R)}
\lesssim 1.
\end{align}
We now claim that there exists some $1\leq J_0\leq J$ such that
\begin{align}
\limsup_{n\to\infty}\|U_n^{J_0}\|_{\diag H_y^{1}(\R)}=\infty.
\end{align}
We argue by contradiction and assume that
\begin{align}\label{uniform bound of unj small}
\limsup_{n\to\infty}\|U_n^j\|_{\diag H_y^{1}(\R)}<\infty\quad\forall\,1\leq j\leq J.
\end{align}
To proceed, we first show that
\begin{align}\label{kkkk uniform bound of unj}
\sup_{J+1\leq k\leq K^*}\lim_{n\to\infty}\bg\|\sum_{j=J+1}^k U_n^j\bg\|_{\diag H_y^{1}(\R)}
\lesssim 1.
\end{align}
Indeed, using triangular inequality, \eqref{smallness aaa} and \eqref{correct1} we immediately obtain
\begin{align}
&\,\sup_{J+1\leq k\leq K^*}\lim_{n\to\infty}\bg\|\sum_{j=J+1}^k \pt_y^{s}U_n^j\bg\|^4_{\diag L_y^2(\R)}\nonumber\\
\leq&\, \sup_{J+1\leq k\leq K^*}\lim_{n\to\infty}\bg(\bg(\sum_{j=J+1}^k
\|U_n^j\|^2_{\diag H_y^{1}(\R)}+\sum_{i,j=J+1,i\neq j}^k\|\pt_y^{s}U_n^i \pt_y^{s} U_n^j\|_{L^2_{t,x}L_y^{1}(\R)}\bg)^2\bg)\lesssim 1
\end{align}
for $s=0,1$. Combining \eqref{kkkk uniform bound of unj} with \eqref{uniform bound of unj small} we deduce that
\begin{align}\label{super uniform}
\sup_{1\leq k\leq K^*}\lim_{n\to\infty}\bg\|\sum_{j=J+1}^k U_n^j\bg\|_{\diag H_y^{1}(\R)}
\lesssim  1.
\end{align}
Therefore, using \eqref{orthog L2} to \eqref{orthog grady}, \eqref{conv of nonlinear profiles in h1} and Strichartz we confirm that the conditions \eqref{cond1 cnls} to \eqref{small error long cnls} are satisfied for sufficiently large $k$ and $n$, where we set $U=U_n$ and $Z=\Psi_n^k$ therein. As long as we can show that \eqref{small error long cnls} is satisfied for $s=\frac23$, we are able to apply Lemma \ref{lem stability cnls} and Lemma \ref{scattering crit} to obtain the contradiction
\begin{align}\label{contradiction 1}
\limsup_{n\to\infty}\|U_n\|_{\diag H_y^{1}(\R)}<\infty.
\end{align}
Direct calculation shows that
\begin{align}
e&=\,i\pt_t\Psi_n^k+\Delta_{x,y}\Psi_n^k+|\Psi_n^k|^{2}\Psi_n^k\nonumber\\
&=\bg(\sum_{j=1}^k (i\pt_tU_n^j+\Delta_{x,y} U_n^j)+|\sum_{j=1}^kU_n^j|^{2}\sum_{j=1}^kU_n^j\bg)
+\bg(|\Psi_n^k|^{2}\Psi_n^k-|\Psi_n^k-e^{it\Delta_{x,y}}w_n^k|^{2}(\Psi_n^k-e^{it\Delta_{x,y}}w_n^k)\bg)\nonumber\\
&=:I_1+I_2.
\end{align}
In the following we show the asymptotic smallness of $I_1$ and $I_2$. Since $U_n^j$ solves \eqref{cnls}, we can rewrite $I_1$ to
\begin{align*}
I_1
=-\bg(\sum_{j=1}^k|U_n^j|^{2}U_n^j-\bg|\sum_{j=1}^kU_n^j\bg|^{2}\sum_{j=1}^kU_n^j\bg)
=O\bg(\sum_{i,j=1,i\neq j}^k|U_n^i|^2|U_n^j|+\sum_{p,i,j=1,i\neq j}^k|U_n^p U_n^i U_n^j|\bg)=:I_{11}+I_{12}.
\end{align*}
We only consider the summand $I_{11}$, the summand $I_{12}$ can be dealt similarly. By H\"older we have
\begin{align}
\||U_n^i|^2|U_n^j|\|_{L_{t,x}^{\frac43}L_y^2}&\lesssim \|U_n^{i}U_n^{j}\|^{\frac12}_{L_{t,x}^2L_y^1}
\|U_n^{i}\|^{\frac32}_{\diag L_y^\infty}\|U_n^{j}\|^{\frac12}_{\diag L_y^\infty}\nonumber\\
&\lesssim \|U_n^{i}U_n^{j}\|^{\frac12}_{L_{t,x}^2L_y^1}
\|U_n^{i}\|^{\frac32}_{\diag H_y^1}\|U_n^{j}\|^{\frac12}_{\diag H_y^1}.
\end{align}
Then \eqref{uniform bound of unj}, \eqref{uniform bound of unj small} and \eqref{smallness aaa} imply
\begin{align}
\lim_{k\to K^*}\lim_{n\to\infty}\| I_{11}\|_{L_{t,x}^{\frac43} L_y^2}=0.
\end{align}
On the other hand,
\begin{align}
\||U_n^i|^2|U_n^j|\|_{L_{t,x}^{\frac43}H_y^1}\lesssim \|U_n^i\|_{\diag H_y^1}^2\|U_n^j\|_{\diag H_y^1}\lesssim1.
\end{align}
Combining with the inequality $\|f\|_{H_y^{\frac23}}\leq\|f\|_{L_y^2}^{\frac13}\|f\|_{H_y^{1}}^{\frac23}$ we infer that
\begin{align}
\lim_{k\to K^*}\lim_{n\to\infty}\| I_{11}\|_{L_{t,x}^{\frac43} H_y^{\frac23}}=0
\end{align}
Next, we prove the asymptotic smallness of $I_2$. Direct calculation shows
\begin{align}
I_2=O\bg(\Psi_n^k(e^{it\Delta_{x,y}}w_n^k)^2+(\Psi_n^k)^2e^{it\Delta_{x,y}}w_n^k+(e^{it\Delta_{x,y}}w_n^k)^3\bg).
\end{align}
But then \eqref{super uniform}, \eqref{interpolation remainder} and Lemma \ref{fractional lemma} immediately yield
\begin{align}
\lim_{k\to K^*}\lim_{n\to\infty}\| I_2\|_{L_{t,x}^{\frac43} H_y^{\frac23}}=0
\end{align}
and Step 2 is complete.
\subsubsection*{Step 3: Reduction to one bad profile and conclusion}
From Step 2 we conclude that there exists some $1\leq J_1\leq K^*$ such that
\begin{align}
\limsup_{n\to\infty}\|U_n^j\|_{\diag H_y^1(\R)}&=\infty\quad \forall \,1\leq j\leq J_1,\label{infinite}\\
\limsup_{n\to\infty}\|U_n^j\|_{\diag H_y^1(\R)}&<\infty\quad \forall \,J_1+1\leq j\leq K^*.
\end{align}
By Lemma \ref{cnls lem large scale proxy} (which is applicable due to \eqref{mo sum}) we deduce that $\ld_\infty^j=1$ for all $1\leq j\leq J_1$. If $J_1>1$, then using \eqref{mo sum}, \eqref{eo sum}, the asymptotic positivity of energies deduced from Step 1 and Lemma \ref{cnls killip visan curve} we know that $\limsup_{n\to\infty}\mD(U_n^1)<\mD^*$, which violates \eqref{infinite} due to the inductive hypothesis. Thus $J_1=1$ and
$$ U_n(0,x)=e^{it_n^1 \Delta_{x,y}}\tdu^1(x-x_n^1)+w_n^1(x).$$
In particular, $\tdu^1\in H_{x,y}^1$. Similarly, we must have $\mM(w_n^1)=o_n(1)$ and $\mH(w_n^1)=o_n(1)$, otherwise we could deduce again the contradiction \eqref{contradiction 1} using Lemma \ref{lem stability cnls}. Combining with Lemma \ref{cnls killip visan curve} we conclude that $\|w_n^1\|_{H_{x,y}^1}=o_n(1)$. Finally, we exclude the cases $t^1_n\to\pm \infty$. We only consider the case $t_n^1\to \infty$, the case $t_n^1 \to -\infty$ can be similarly dealt. Indeed, using Strichartz we obtain that
\begin{align}
\|e^{it\Delta_{x,y}}U_n(0)\|_{\diag H_y^1([0,\infty))}\lesssim
\|e^{it\Delta_{x,y}}\tdu^1\|_{\diag H_y^1([t^1_n,\infty))}+\|w_n^1\|_{H^1}\to 0
\end{align}
and using Lemma \ref{lemma cnls well posedness} we infer the contradiction \eqref{contradiction 1} again. This completes the desired proof.
\end{proof}

\begin{lemma}[Existence of a minimal blow-up solution]\label{category 0 and 1}
Suppose that $\mD^*<\infty$. Then there exists a global solution $U_c$ of \eqref{cnls} such that $\mD(u_c)=\mD^*$ and
\begin{align}
\|U_c\|_{\diag H_y^1((-\infty,0])}=\|u_c\|_{\diag H_y^1([0,\infty))}=\infty.
\end{align}
Moreover, $U_c$ is almost periodic in $H_{x,y}^1$ modulo $\R_x^2$-translations, i.e. the set $\{U(t):t\in\R\}$ is precompact in $H_{x,y}^1$ modulo translations w.r.t. the $x$-variable.
\end{lemma}

\begin{proof}
As discussed at the beginning of this section, under the assumption $\mD^*<\infty$ one can find a sequence $(U_n)_n$ of solutions of \eqref{cnls} that satisfies the preconditions of Lemma \ref{Palais Smale}. We apply Lemma \ref{Palais Smale} to infer that $(U_n(0))_n$ (up to modifying time and space translation) is precompact in $H_{x,y}^1$. We denote its strong $H_{x,y}^1$-limit by $\psi$. Let $U_c$ be the solution of \eqref{cnls} with $U_c(0)=\psi$. Then $\mD(U_c(t))=\mD(\psi)=\mD^*$ for all $t$ in the maximal lifespan $I_{\max}$ of $U_c$ (recall that $\mD$ is a conserved quantity).

We first show that $U_c$ is a global solution. We only show that $s_0:=\sup I_{\max}=\infty$, the negative direction can be similarly proved. If this does not hold, then by Lemma \ref{lemma cnls well posedness} there exists a sequence $(s_n)_n\subset \R$ with $s_n\to s_0$ such that
\begin{align*}
\lim_{n\to\infty}\|U_c\|_{\diag H_y^1((-\inf I_{\max},s_n])}=\lim_{n\to\infty}\|u_c\|_{\diag H_y^1([s_n,s_0))}=\infty.
\end{align*}
Define $V_n(t):=u_c(t+s_n)$. Then \eqref{precondition} is satisfied with $t_n\equiv 0$. We then apply Lemma \ref{Palais Smale} to the sequence $(V_n(0))_n$ to conclude that there exists some $\varphi\in H_{x,y}^1$ such that, up to modifying the space translation, $U_c(s_n)$ strongly converges to $\varphi$ in $H_{x,y}^1$. But then using Strichartz we obtain
\begin{align*}
\|e^{it\Delta_{x,y}}U_c(s_n)\|_{\diag H_y^1([0,s_0-s_n))}=\|e^{it\Delta_{x,y}}\varphi\|_{\diag H_y^1([0,s_0-s_n))}+o_n(1)=o_n(1).
\end{align*}
By Lemma \ref{lemma cnls well posedness} we can extend $U_c$ beyond $s_0$, which contradicts the maximality of $s_0$. Now by \eqref{oo1} and Lemma \ref{lem stability cnls} it is necessary that
\begin{align}\label{blow up uc}
\|U_c\|_{\diag H_y^1((-\infty,0])}=\|U_c\|_{\diag H_y^1([0,\infty))}=\infty.
\end{align}

We finally show that the orbit $\{U_c(t):t\in\R\}$ is precompact in $H_{x,y}^1$ modulo ($\R^2_x$)-translations. Let $(\tau_n)_n\subset\R$ be an arbitrary time sequence. Then \eqref{blow up uc} implies
\begin{align*}
\|U_c\|_{\diag H_{y}^1((-\infty,\tau_n])}=\|U_c\|_{\diag H_{y}^1([\tau_n,\infty))}=\infty.
\end{align*}
The claim follows by applying Lemma \ref{Palais Smale} to $(U_c(\tau_n))_n$.
\end{proof}

\subsection{Extinction of the minimal blow-up solution}\label{Extinction of the minimal blow-up solution}
We exclude in this final section the minimal blow-up solution that we deduced from the last section. The following lemma is an immediate consequence of the fact that $U_c$ is almost periodic in $H_{x,y}^1$ and conservation of momentum. The proof is standard, we refer to \cite{non_radial} for details of the proof.
\begin{lemma}\label{holmer}
Let $U_c$ be the minimal blow-up solution given by Lemma \ref{category 0 and 1}. Then there exists some function $x:\R\to\R^2$ such that
\begin{itemize}

\item[(i)] For each $\vare>0$ there exists $R>0$ so that
\begin{align}
\int_{|x+x(t)|\geq R}|\nabla_{x,y} U_c(t)|^2+|U_c(t)|^2+|U_c(t)|^4\,dxdy\leq\vare\quad\forall\,t\in\R.
\end{align}


\item[(ii)] The center function $x(t)$ obeys the decay condition $x(t)=o(t)$ as $|t|\to\infty$.
\end{itemize}
\end{lemma}

\begin{proof}[Proof of Theorem \ref{main thm}]
We will show the contradiction that the minimal blow-up solution $U_c$ given by Lemma \ref{category 0 and 1} is equal to zero, which will finally imply Theorem \ref{main thm}. First we notice that since $U$ is a non-zero almost-periodic solution in $H_{x,y}^1$, we have
\begin{align}
\inf_{t\in\R}\|\nabla_{x}U_c(t)\|_{L_{x,y}^2}^2=:\rho>0.
\end{align}
Next, let $\chi:\R^2\to\R$ be a smooth radial cut-off function satisfying
\begin{align*}
\chi=\left\{
             \begin{array}{ll}
             |x|^2,&\text{if $|x|\leq 1$},\\
             0,&\text{if $|x|\geq 2$}.
             \end{array}
\right.
\end{align*}
Then for $R>0$, we define the local virial action $z_R(t)$ by
\begin{align*}
z_{R}(t):=\int R^2\chi\bg(\frac{x}{R}\bg)|U_c(t,x)|^2\,dxdy.
\end{align*}
Direct calculation yields
\begin{align}
\pt_t z_R(t)=&\,2\,\mathrm{Im}\int R\nabla_x \chi\bg(\frac{x}{R}\bg)\cdot\nabla_x U_c(t)\overline{U}_c(t)\,dxdy,\label{final4}\\
\pt_{tt} z_R(t)=&\,4\int \pt^2_{x_j x_k}\chi\bg(\frac{x}{R}\bg)\pt_{x_j} U_c\pt_{x_k}\overline{U}_c\,dxdy-\frac{1}{R^2}\int\Delta_x^2\chi\bg(\frac{x}{R}\bg)|U_c|^2\,dxdy
-\int\Delta_x\chi\bg(\frac{x}{R}\bg)|U_c|^4\,dxdy.
\end{align}
We then obtain that
\begin{align}
\pt_{tt} z_R(t)=16\mH_*(U(t))+A_R(U_c(t)),
\end{align}
where $\mH_*$ is defined by \eqref{thres no energy} and
\begin{align*}
A_R(U_c(t))=&\,4\int\bg(\pt^2_{x_j}\chi\bg(\frac{x}{R}\bg)-2\bg)|\pt_{x_j} U_c|^2\,dxdy+4\sum_{j\neq k}\int_{R\leq|x|\leq 2R}\pt_{x_j}\pt_{x_k}\chi\bg(\frac{x}{R}\bg)\pt_{x_j} U\pt_{x_k}\overline{U}_c\,dxdy\nonumber\\
&\,-\frac{1}{R^2}\int\Delta_x^2\chi\bg(\frac{x}{R}\bg)|U_c|^2\,dxdy
-\int\bg(\Delta_x\chi\bg(\frac{x}{R}\bg)-4\bg)|U_c|^4\,dxdy.
\end{align*}
We have the rough estimate
\begin{align*}
|A_R(u(t))|\leq C_1\int_{|x|\geq R}|\nabla_x U_c(t)|^2+\frac{1}{R^2}|U_c(t)|^2+|U_c(t)|^4\,dxdy
\end{align*}
for some $C_1>0$. By Lemma \ref{cnls killip visan curve} we know that there exists some $\beta>0$ such that \eqref{threshold7} and \eqref{threshold8} hold for this $\beta$. By \eqref{thres no energy} we deduce that there exists some $c_\beta>0$ such that
\begin{align}\label{small extinction ff}
16\mH_*(U(t))\geq c_\beta\|\nabla_{x}U_c(t)\|_{L_{x,y}^2}^2\geq c_\beta\rho=:2\eta_1>0.
\end{align}
From Lemma \ref{holmer} it follows that there exists some $R_0\geq 1$ such that
\begin{align*}
\int_{|x+x(t)|\geq R_0}|\nabla_{x,y} U_c(t)|^2+|U_c(t)|^2+|U_c(t)|^4\,dxdy\leq \frac{\eta_1}{C_1}.
\end{align*}
Thus for any $R\geq R_0+\sup_{t\in[t_0,t_1]}|x(t)|$ with some to be determined $t_0,t_1\in[0,\infty)$, we have
\begin{align}\label{final3}
\pt_{tt} z_R(t)\geq \eta_1
\end{align}
for all $t\in[t_0,t_1]$. By Lemma \ref{holmer} we know that for any $\eta_2>0$ there exists some $t_0\gg 1$ such that $|x(t)|\leq\eta_2 t$ for all $t\geq t_0$. Now set $R=R_0+\eta_2 t_1$. Integrating \eqref{final3} over $[t_0,t_1]$ yields
\begin{align}\label{12}
\pt_t z_R(t_1)-\pt_t z_R(t_0)\geq \eta_1 (t_1-t_0).
\end{align}
Using \eqref{final4}, Cauchy-Schwarz and Lemma \ref{cnls killip visan curve} we have
\begin{align}\label{13}
|\pt_t z_{R}(t)|\leq C_2 \mD^*R= C_2 \mD^*(R_0+\eta_2 t_1)
\end{align}
for some $C_2=C_2(\mD^*)>0$. \eqref{12} and \eqref{13} give us
\begin{align*}
2C_2 \mD^*(R_0+\eta_2 t_1)\geq\eta_1 (t_1-t_0).
\end{align*}
Setting $\eta_2=\frac{\eta_1}{4C_2\mD^*}$, dividing both sides by $t_1$ and then sending $t_1$ to infinity we obtain $\frac{1}{2}\eta_1\geq\eta_1$, which implies $\eta_1=0$, a contradiction. This completes the proof.
\end{proof}

\subsubsection*{Acknowledgments}
The author acknowledges the funding by Deutsche Forschungsgemeinschaft (DFG) through the Priority Programme SPP-1886 (No. NE 21382-1).

\addcontentsline{toc}{section}{References}
\bibliographystyle{hacm}

\end{document}